\documentclass[12pt,reqno]{amsart}
\usepackage{hyperref}
\usepackage{amsmath}
 \usepackage{bm,amsthm,amsfonts,latexsym,amssymb,commath} 

\usepackage{siunitx}

\usepackage{graphicx,float,multicol}

\newtheorem{lemma}{\bf{Lemma} }[section]
\newtheorem{proposition}{\bf{Proposition}}[section]
\newtheorem{theorem}{\bf{Theorem}}[section]
\newtheorem{remark}{\sc{Remark} }[section]
\newtheorem{definition}{\sc{Definition} }[section]
\newtheorem{corollary}{\bf{Corollary} }[section]

\numberwithin{equation}{section}
\begin{document}

\title{On the smallness conditions for a PEMFC single cell problem}
\author{Luisa Consiglieri}
\address{Luisa Consiglieri, Independent Researcher Professor, European Union}
\urladdr{\href{http://sites.google.com/site/luisaconsiglieri}{http://sites.google.com/site/luisaconsiglieri}}

\begin{abstract} 
The aim of the present paper is to prove whose smallness conditions being necessary in order to get the final result of existence of a solution. 
In the first part,  we present the model for a proton exchange membrane  fuel cell (PEMFC) single cell and we clarify the interactions of the different components
 namely,  velocity,  pressure,  density,  temperature and potential.
The final mathematical model is a quasilinear elliptic system where the cross effects have a strong interlink. 
It consists of the Stokes--Darcy system altogether with thermoelectrochemical system under some non-standard interface and boundary conditions.
The proof of existence of weak solutions relies on the Tychonof fixed point theorem, by providing some regularity and some smallness conditions.
The actual system is divided into two systems of equations and they are separately studied. 
The novelty of the present work is to establish quantitative estimates for improving the technical hypotheses and, in particular, the smallness conditions
in the two-dimensional  case.
Indeed, the smallness conditions only can be explicit if quantitative estimates are established.
To this aim, we also establish quantitative estimates for the Poincar\'e and Sobolev inequalities and for some trilinear terms.
\end{abstract}

\keywords{PEM fuel cell; multiregion domain; Stokes--Darcy system; Beavers--Joseph--Saffman boundary condition, thermoelectrochemistry.}
\subjclass[2020]{Primary: 76S05, 80A50; Secondary: 35Q35, 35Q79.}
\maketitle

\section{Introduction}
\label{int} 

In this paper, we present and study a model for  proton exchange membrane (PEM) fuel cells, those that work
 at low operating temperature such as the  polymer electrolyte membrane fuel cells with hydrogen supply (H\(_2\)PEMFC) and direct methanol fuel cells (DMFC).
PEM fuel cells have been object of study in the last decades by their inherent energy conversion.
They possess functional structure from the nanoscale up to the macroscale (see \cite{wetton,seca} and the references therein)
and then their descriptive models are multiscale thermoelectrochemical (TEC) systems.
Numerical simulations  have often been implemented in the past two decades for the study of different tasks performance
 \cite{gosh,gurau,HuYu2004,Omran,KangBuchi,singh} and, in particular, for computational fluid dynamics (CFD), see  \cite{Iranzo,zhang2022} and the references therein.
Also experimental works have been performed, see \cite{hawaii} and the references therein.
A  simplified  model of a self-humidifying PEM fuel cell is both numerically simulated and experimentally tested in \cite{saleh}.

The fuel cell consists of a membrane,  two electrodes and two flow regions.
 The membrane is a porous medium, which   is electron insulating and serves to conduct ions  produced at one electrode to the
other, namely the ionic charge carrier of H\(_3\)O\(^+\) in particular for H\(_2\)PEMFC or DMFC.

The mathematical model firstly consists of coupling of the Stokes--Fourier and Darcy--Fourier equations, known as the Stokes--Darcy--Fourier (SDF)  system.
We refer to \cite{cms2012} the study of the SDF system under both Beavers--Joseph--Saffman (BJS) and Beavers--Joseph (BJ)
 interface boundary conditions.
The generalization of BJS-SDF problem to non-Newtonian fluids is studied in \cite{rse2013}
by introducing the Forchheimer model.
Other approach is introduced in \cite{fnt2010}, in which a nonlinear Darcy's law is obtained by asymptotic limit
of solutions to the Navier--Stokes--Fourier system  in perforated domains
with tiny holes, where the diameter of the holes is proportional to their mutual distance, by
homogenization method. 

Secondly,  a thermoelectrochemical model is gathered to the Beavers--Joseph--Saff\-man/Stokes--Darcy problem, with some modified Butler--Volmer interface condition.
The Joule effect is taken into account  on the energy equation due to the electrical current.
To assure that the Joule effect works better than  a \(L^1\) data,
we provide some elliptic regularity for \(q>n=2\) (space dimension) as it has been used in real world problems
(see \cite{gro89,groreh} and references therein). Recently, in  \cite{jo2016,jo2021} the elliptic regularity is studied  and improved via
 the quantitative Sneiberg inequality.

Here, we do not assume that the mathematically inconvenient constants are equal to one,
because the magnitude of each constant is physically relevant.
Other important physical behavior is the discontinuous coefficients to allow, for instance, the viscosities being temperature dependent.
This gives an extra draw back to the elliptic system.
It is known that the fixed point argument is the primordial shortcoming
in the existence of solutions of nonlinear PDE at the steady state and
some smallness conditions are required for the application of  the fixed point argument.
  Several hypotheses are made on the coefficients in the equations. 
Some of them are  natural but others are  technical. 
 The reason being that the mathematical model has a strong interlink due to the cross effects.
Future work should be done to improve the smallness conditions.

The outline of the present paper is as follows.
Next section,  we  introduce the mathematical equations of the concrete physical model under consideration at the steady state. 
In Section \ref{framework}, we state the set of hypothesis and the two-dimensional (2D) main result. 
Also, the physical meaning of the assumptions is  discussed for a  H\(_2\)-PEMFC.
In Section \ref{strat}, we  delineate the strategy used in this paper, 
namely the actual system is divided into two systems of equations, which are separately studied, in order to use
 a fixed point argument.
In order to be able to use this machinery we establish some auxiliary results in Section \ref{tdt}.
Then, Section \ref{auxtdt} is devoted to the existence of the two auxiliary problems, where a special care is taken in determining the quantitative estimates.
Finally, Section \ref{fpthm} is concern to the proof of the main result.

\section{Statement of the fuel cell problem}
\label{spemfc}

Let \(\Omega\) be a bounded multidomain   of \(\mathbb{R}^n\), \(n\geq 2\), that is,
 the domain \(\Omega\) is  a connected open set, which consists of different pairwise disjoint  Lipschitz subdomains.
Precisely, it is separated into five regions,  
\({\Omega}_\mathrm{fuel} \), \({\Omega}_\mathrm{a} \),  \({\Omega}_\mathrm{m} \), \({\Omega}_\mathrm{c} \) and \({\Omega}_\mathrm{air} \),
with total width  \(W=2 l_f + 2l_a+l_m\).
The multidomain \(\Omega\) represents one single PEM fuel cell,
 which its 2D (\(xy\) cross-section) representation is schematically illustrated in Figure~\ref{fpem}.
Moreover, it has the   membrane interface \(\Gamma_\mathrm{CL}=\Gamma_\mathrm{a}\cup\Gamma_\mathrm{c}\) and
the porous-fluid boundary \(\Gamma\),  of \((n-1)\)-dimensional Lebesgue measure.
\begin{figure}
 \includegraphics[width=0.5\textwidth]{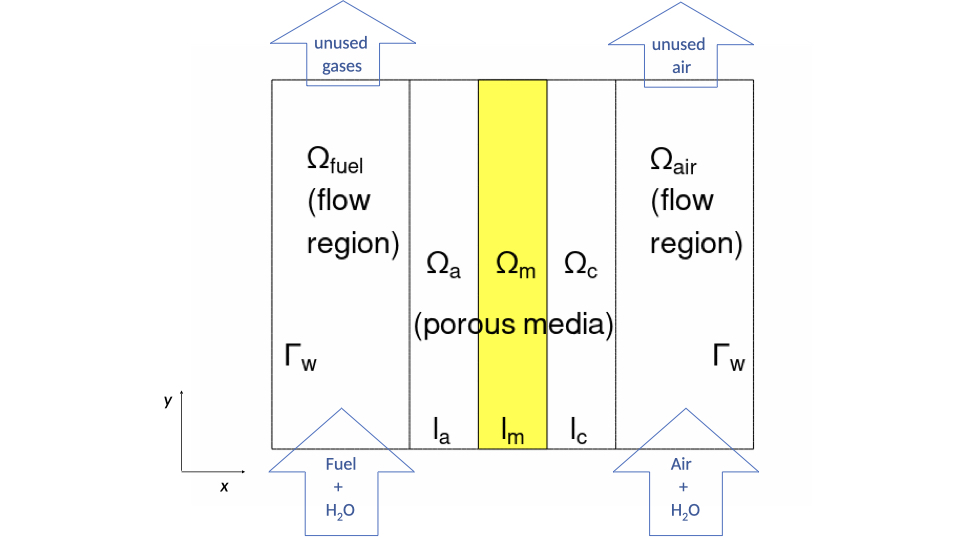} 
\caption{The flow region \(\Omega_\mathrm{f}= \Omega_\mathrm{fuel}\cup \Omega_\mathrm{air}\) and
 the porous region \(\Omega_\mathrm{p}
= \Omega_\mathrm{a}\cup\overline{\Omega}_\mathrm{m}\cup\Omega_\mathrm{c}\)
(not in scale), with width \(l_\mathrm{a}+l_\mathrm{m}+l_\mathrm{c}<< L\) 
where \(L=1-10\,\si{\centi\metre}\) denotes each channel length. }\label{fpem}
\end{figure}

We call by the fluid bidomain  \({\Omega}_\mathrm{f} \) the two channels,
namely the anodic fuel channel \(\Omega_\mathrm{fuel}\) and the cathodic air channel \(\Omega_\mathrm{air}\).
Each channel has a typical characteristic length \(l_f=\SI{0.001}{\metre}\) \cite{li2011,zhang2022}.

We call by the porous domain  \({\Omega}_\mathrm{p} \) 
 the proton conducting membrane  \(\Omega_\mathrm{m}\),  the anode and cathode backing  layers,  \({\Omega}_\mathrm{a} \) and \({\Omega}_\mathrm{c} \),
and  the anode and cathode catalyst layers (CL),  \(\Gamma_\mathrm{a}\) and \(\Gamma_\mathrm{c}\), respectively.
The backing layers are porous gas diffusion layers (GDL), with fuel in the anodic compartment and air in the cathodic compartment,where
the traveling of the free electrons occurs and a current collector \(\Gamma_\mathrm{cc}\) is attained.
The catalyst layers have negligible measure when compared with the backing layers
 (the backing layers are approximately \(l_a=l_c=\SI{200}{\micro\metre}\)  in thickness,
 while the catalyst layers are \SIrange{5}{10}{\micro\metre} \cite{li2011,zhang2022}), and then they are assumed to be
  interfaces between the membrane separator and the backing layers.
Hereafter,  the subscripts, a and c, stand for anode and  cathode, respectively.

The porous-fluid boundary is the interface  \(\Gamma=\partial\Omega_\mathrm{f}\cap\Omega=\partial\Omega_\mathrm{p}\cap\Omega\).

\subsection{In the fluid bidomain  \({\Omega}_\mathrm{f}= \Omega_\mathrm{fuel}\cup \Omega_\mathrm{air} \)}

By the characteristics of the channels, the convection for fluid and heat flows may be neglected.

The governing equations are the conservation of mass, momentum,  species
and energy, a.e. in \(\Omega_\mathrm{f}\),
\begin{align}
\nabla\cdot(\rho \mathbf{u})=0; \label{mass}\\
 \nabla  \cdot\tau = \nabla p ; \label{momentum}\\
\nabla\cdot(\mathbf{u} \rho_i)+\nabla\cdot \mathbf{j}_i=0; \label{species}\\
\label{heateqs}
\nabla\cdot \mathbf{q} =0,
\end{align}
for the uncharged species \(i=1,\cdots,\mathrm{I}\).
 The unknown functions are the density \(\rho\), the velocity \(\mathbf{u}=(u_x,u_y,u_z)\),
 the mass concentration vector \(\bm{\rho}=(\rho_1,\cdots, \rho_\mathrm{I})\) and the temperature \(\theta\).
Each  partial density is defined by
\begin{equation}\label{defrhoi}
\rho_i = M_ic_i,
\end{equation}
where 
\( M_i\) denotes the molar mass [\si{\kilogram\per\mole}] and 
\(c_i\) is the  molar concentration [\si{\mole\per\cubic\metre}] of the species  \(i\).
The following values are known: \(M(\mathrm{H}_2\mathrm{O})= \SI{18}{\gram\per\mole}\),
 \(M(\mathrm{O}_2)=\SI{32}{\gram\per\mole}\)  and \(M(\mathrm{H}_3\mathrm{O}^+)=\SI{1}{\gram\per\mole}\).
For the H\(_2\)PEMFC, \(M(\mathrm{H}_2)= \SI{2}{\gram\per\mole}\),  while for DMFC, \(M(\mathrm{CH}_4\mathrm{O}) = \SI{32}{\gram\per\mole}\).

The deviatoric stress tensor  \( \tau\), which is temperature dependent, obeys the constitutive law 
\begin{equation}\label{diff}
\tau =\mu(\theta)D\mathbf{u}+\lambda(\theta) \mathrm{tr}(D\mathbf{u}) \mathsf{I},
\qquad \mathrm{tr}(D\mathbf{u})=\mathsf{I}:D\mathbf{u}= \nabla\cdot\mathbf{u},
\end{equation}
where \( D=(\nabla+\nabla^T)/2\) denotes the symmetric gradient and  \( \mathsf{I}\) denotes the identity (\(n\times n\))-matrix.
The viscosity coefficients  \( \mu\) and \( \lambda\) are   in accordance with the second law of thermodynamics
\begin{equation}\label{2law}
\mu(\theta)>0,\quad \nu(\theta) :=\lambda(\theta)+\mu(\theta)/ n \geq 0,
\end{equation}
with \(\nu\) denoting the bulk (or volume) viscosity  and \(\mu/2\) being the shear (or dynamic) viscosity.
Taking into account the convention on implicit summation over repeated indices,  we denote \(\zeta:\varsigma=\zeta_{ij}\varsigma_{ij}\).

We assume that the anode and cathode gas mixtures with water vapor act as ideal gases \cite{li2011}, that is,
the pressure \(p\) obeys the Boyle--Marriotte law
\begin{equation}\label{boyle}
p = R_M\rho\theta,
\end{equation} 
where \( R_M=R/M\) with \( M\) denoting the molar mass [\si{\kilogram\per\mole}].

The  phenomenological fluxes,  \(\mathbf{j}_{i}\)  [\si{\kilogram\per\second\per \square\metre}] and \(\mathbf{q}\) [\si{\watt\per \square\metre}],
  are explicitly driven by 
\begin{align}\label{defjif}
\mathbf{j}_{i} &= - D_{i} (\theta) \nabla \rho_{i} -
 \sum_{\genfrac{}{}{0pt}{2}{j=1}{j\not= {i}}}^\mathrm{I} D_{{i}j} (\theta) \nabla \rho_j -\rho_{i}S_{i}( c_{i} ,\theta) \nabla\theta  
; \\ \label{defqf}
\mathbf{q} &= -R\theta^2 \sum_{j=1}^\mathrm{I} D'_j ( c_j, \theta) \nabla c_j -k (\theta) \nabla\theta ,
\end{align}
with \(i=1,\cdots,\mathrm{I} \), see \cite{lap2017} and the references therein.
These include
the Fick law (with the diffusion coefficient \(D_i\) [\si{\metre\squared\per\second}]),
  the Fourier law (with  the thermal conductivity  \(k\)  [\si{\watt\per\metre\per\kelvin}]),
and the Dufour--Soret  cross effect (with the Dufour coefficient \(D'_i\)  [\si{\metre\squared\per\second\per\kelvin}]
 and the Soret coefficient \(S_i\)   [\si{\metre\squared\per\second\per\kelvin}]).
While in binary liquid mixtures the Dufour  effect is negligible, 
in binary gas mixtures the Dufour effect can be significant \cite{HollingerLucke1995}.
The universal constant is the so-called the gas constant
 \(R= \SI{8.314}{\joule\per\mole\per\kelvin}\).

 Hereafter the subscript \(i\) stands for the correspondence to the ionic component \(i=1,\cdots,\mathrm{I}\) 
intervened in the reaction process, with \(\mathrm{I}\in\mathbb{N}\) being either
 \(\mathrm{I}_\mathrm{p}\) whenever \(\Omega_\mathrm{p}\) or  \(\mathrm{I}_\mathrm{f}\) whenever \(\Omega_\mathrm{f}\).
For the sake of simplicity,  we consider the number of species
\(\mathrm{I} = \mathrm{I}_\mathrm{a} =\mathrm{I}_\mathrm{m} =\mathrm{I}_\mathrm{c} =2\)  (cf. Table~\ref{table1}).
 \begin{table}[h]
 \begin{tabular}{|c|c|c|c|}
\hline  \(i\)&  \(\Omega_\mathrm{fuel}\cup\Omega_\mathrm{a}\) & \(\Omega_\mathrm{m}\)  &  \(\Omega_\mathrm{air}\cup\Omega_\mathrm{c}\)    
 \\
\hline 
1 & fuel & H\(_3\)O\(^+\) &   O\(_2\) \\ 
2 &H\(_2\)O &H\(_2\)O  &H\(_2\)O  \\ 
\hline 
\end{tabular} 
\caption{The correspondence of each component to each region}\label{table1}
 \end{table}

 The water  is present in  fluid and vapor  states, and in both cases it can be modeled as
a Newtonian fluid (linearly viscous fluid).

\subsection{In the porous domain  \( \Omega_\mathrm{p}=\Omega_\mathrm{a}\cup\overline{\Omega}_\mathrm{m}\cup\Omega_\mathrm{c}\)}

The governing equations,  after a volume averaging procedure, are
\begin{align}
\nabla  \cdot \mathbf{u}_\mathrm{D} &=0 
; \label{momentump}\\
\nabla\cdot \mathbf{j}_i &=0; \label{speciesp}
 \\ \label{heatp}
\nabla\cdot \mathbf{q}  &=Q, \mbox{ a.e. in } \Omega_\mathrm{a}\cup\Omega_\mathrm{m}\cup\Omega_\mathrm{c},
\end{align}
for   \(i=1,2 \)  and \(\mathrm{I}=2\),  according to Table~\ref{table1}.
Here, it   is omitted the bracket \(\langle\cdot  \rangle\),  which usually represents  the volume averaged. Thus,
 the temperature \(\theta\) is the spatially averaged (over a representative elementary volume) microscopic quantity, 
and the Darcy velocity \(\mathbf{u}_\mathrm{D}\) [\si{\metre\per\second}] is the superficial average quantity.

 The volume averaged density  \(\rho\) 
 of the fluid is piecewise  constant, \(\rho_\mathrm{water}=\SI{970}{\kilogram\per\cubic\metre}\) in \(\Omega_\mathrm{m}\)
and \(\rho_\mathrm{air}=\SI{0.995}{\kilogram\per\cubic\metre}\)  in \(\Omega_\mathrm{a}\cup\Omega_\mathrm{c}\),
  due to \(\rho_\mathrm{air} =p_\mathrm{atm}M_\mathrm{air}/(R\theta_r)\), 
at the typical operating temperature of \(\theta_r = \SI{357.15}{\kelvin}\) (= \SI{84}{\celsius}),
\(p_\mathrm{atm} =\SI{101.325}{\kilo\pascal}\) and
\(M_\mathrm{air} =\SI{28.97}{\gram\per\mole}\).

The  Darcy velocity \(\mathbf{u}_\mathrm{D}\) obeys
\begin{equation}\label{darcyg}
\mu\mathbf{u}_\mathrm{D}=- K_g \nabla p 
\end{equation}
where 
\(p\)  is the intrinsic average  pressure [\si{\pascal}] and
 \(\mu=\mu(\theta)\) denotes the viscosity [\si{\pascal\second}].  In the Darcy equation \eqref{darcyg}, the gravity is neglected and
 \(K_g\) represents the gas permeability [\si{\metre\squared}]. It is known that the gas permeability depends on the fiber diameter,
and the Carman--Kozeny equation is commonly used \cite{seca}.
The permeability  should include Klinkenberg effect due to the behavior of
gas flow in porous media, \textit{i.e.}  it obeys the Klinkenberg equation
 \begin{equation}\label{defKg}
 K_g=K_l\left(1+\frac{b}{p}\right).
 \end{equation}
The Klinkenberg correction \(b>0\) in \(\Omega_\mathrm{a}\cup\Omega_\mathrm{c}\)
 depends on space and temperature through the porosity,  \(b=0\) in \(\Omega_\mathrm{m}\),
and \(K_l>0\) being the liquid permeability of the porous media  that only depends on the porosity. Therefore,  \(b\geq 0\) and \(K_l\) are constant.

The  phenomenological fluxes,  \(\mathbf{j}_{i}\)  [\si{\kilogram\per\second\per \square\metre}] and \(\mathbf{q}\) [\si{\watt\per \square\metre}],
  are explicitly driven by 
\begin{align}\label{defji}
\mathbf{j}_{i} &= - D_{i} (\theta) \nabla \rho_{i} -D_{{i}j} (\theta) \nabla \rho_j -\rho_{i}S_{i}( c_{i} ,\theta) \nabla\theta  -u_{i} \rho_{i}\nabla\phi
; \\ \label{defq}
\mathbf{q} &= -R\theta^2 \sum_{j=1}^2 D'_j ( c_j, \theta) \nabla c_j -k (\theta) \nabla\theta -\Pi (\theta) \sigma (\mathbf{c},\theta) \nabla\phi ,
\end{align}
with \(i,j=1,2 \), and \(j\not= i\).
The  phenomenological fluxes are explicitly driven by the gradients of  the temperature  \(\theta\) and
 the mass concentration vector \(\bm{\rho}\), in the form (up to some temperature and concentration dependent factors)
as in \eqref{defjif}-\eqref{defqf}, altogether  by the gradient of  the electric potential \(\phi\),
by incorporating
the Peltier--Seebeck cross effect. We remind that the Peltier coefficient \(\Pi\) [\si{\volt}] and the Seebeck coefficient 
 \(\alpha_\mathrm{S}\)  [\si{\volt\per\kelvin}]  are correlated by the first Kelvin relation
\begin{equation}\label{kelvin}
\Pi(\theta) =\theta\alpha_\mathrm{S}(\theta).
\end{equation}

For  the ionic component \(i=\) H\(_3\)O\(^+\),  the proton flux \(\mathbf{J}_i
=\mathbf{j}_i/M_i\) [\si{\mole\per\second\per \square\metre}]   obeys \eqref{defji}  in \(\Omega_\mathrm{m}\),
where in the first term \(D_i=\kappa/(z_iF)\), with
 the proton ionic conductivity \(\kappa\) being no constant in accordance with the membrane did not being fully hydrated. 
The universal constant is the so-called Faraday constant  \(F= \SI{9.6485e4}{\coulomb\per\mole}\).

For  the dissolved water \(i=\) H\(_2\)O, the molar flux \(\mathbf{J}_i\) obeys \eqref{defji} in \(\Omega_\mathrm{m}\),
 where the second term  means the electro-osmosis (\(j\not= i\)), 
 with  \(D_{ij}=n_\mathrm{d}\) representing the electro-osmostic drag coefficient \cite{li2011}.
Moreover, \(z_{\mathrm{H}_2\mathrm{O}} =0\) and \(u_2=0\).
 
The  Darcy velocity as a drift velocity does not have the relevance as in \eqref{species}, and the drift term may be neglected in \eqref{speciesp}.
Indeed, the drift velocity   appears  in the last term in \eqref{defji} as
\[
u_i\mathbf{E} \quad (\mathbf{E}=-\nabla\phi),
\]
where \(\mathbf{E}\) stands for electric field strength in \(\Omega_\mathrm{m}\)
and the ionic mobility \(u_i\)  [\si{\metre\squared\per\second\per\volt}] satisfies the Nernst--Einstein relation 
\begin{equation}\label{mobility}
u_i= |z_i| F D_i/(R\theta) =\kappa(\theta)/(R\theta),
\end{equation}
which does not vanish for the valence of species \(z_{\mathrm{H}^+}=1\).

In the energy equation \eqref{heatp},  the Joule effect 
\begin{equation}\label{joule}
Q=\chi_{\Omega_\mathrm{a}\cup\Omega_\mathrm{c}} \sigma(\mathbf{c},\theta)  | \nabla\phi|^2
\end{equation}
takes into account that the effect of flow velocity is negligible when
compared to the electrical current that exists in \(\Omega_\mathrm{a}\cup\Omega_\mathrm{c}\).
 
 The electric current density \(\mathbf{j}\)  [\si{\ampere\per \square\metre}]
is given by the Ohm law  (with  the electrical conductivity  \(\sigma\)  [\si{\siemens\per\metre}])
\begin{equation}\label{ohm}
\mathbf{j} = -\sigma(\mathbf{c},\theta) \nabla\phi \quad\mbox{  in }\Omega_\mathrm{a}\cup\Omega_\mathrm{c},
\end{equation}
and it verifies
\begin{equation}\label{electric}
\nabla\cdot\mathbf{j}=0 \quad\mbox{ a.e. in } \Omega_\mathrm{a}\cup\Omega_\mathrm{c} .
\end{equation}
In the fuel cell model,  the electric potential is given at the membrane interface (cf.  Subsection \ref{interface}).
Notice that there is no  electric current density in \(\Omega_\mathrm{m}\), \textit{i.e.} 
 the electric flux  \(\mathbf{j}\)  is the ionic current density \( \mathbf{j}_\mathrm{m}\)  that verifies 
\(\mathbf{j}_\mathrm{m}= z_{\mathrm{H}^+} F \mathbf{J}_{\mathrm{H}^+}\). In practice, the flow indeed obeys the constitutive law
\begin{equation}
\mathbf{j}_m = -\kappa (\theta) \nabla \rho_1/M_1 -\alpha_\mathrm{S} (\theta) \sigma _m(\rho_2,\theta) \nabla\theta -\sigma_m(\rho_2,\theta) \nabla\phi ,
\end{equation}
where the parameters are well determined. 
For instance,   the proton conductivity \(\sigma_m\) [\si{\siemens\per\metre}] may be water content and temperature dependent in contrast with the 
ionomer Nafion constant assumed in  \cite{pemfc}.

\subsection{On the outer boundary \(\partial\Omega\)}

The boundary of \(\Omega\)   is constituted by three pairwise disjoint open \((n-1)\)-dimensional sets, namely 
 \(\Gamma_\mathrm{in}\),  \(\Gamma_\mathrm{out}\) and \(\Gamma_\mathrm{w}\)
which represent the inlet, outlet and wall boundaries, respectively, 
\[
\partial\Omega= \Gamma_\mathrm{in}\cup \Gamma_\mathrm{out}\cup  \overline \Gamma_\mathrm{w}.
\]
The wall boundary has a subpart  \(\Gamma_\mathrm{cc}\subset \partial\Omega_\mathrm{p}\)
 that stands for the  current collector, meaning that the remaining  wall boundary is electrical current  insulated.
 The inlet and outlet sets are the union of two disjoint connected open  \((n-1)\)-dimensional sets, namely,
 \begin{align*}
 \Gamma_\mathrm{in} &= \Gamma_\mathrm{in,a}\cup  \Gamma_\mathrm{in,c} ; \\
 \Gamma_\mathrm{out} &= \Gamma_\mathrm{out,a}\cup  \Gamma_\mathrm{out,c},
\end{align*}
corresponding to the anodic and cathodic channels, \(\Omega_\mathrm{fuel}\) and \(\Omega_\mathrm{air}\) (cf. Figure~\ref{fpem}). 

On the wall boundary \(\Gamma_\mathrm{w} \), 
the no outflow boundary conditions are considered to the velocity  and the species,  
\begin{equation}
\mathbf{u}\cdot\mathbf{n}=(\rho_i \mathbf{u} + \mathbf{j}_i) \cdot\mathbf{n}=0 
\quad (i=1,\cdots,\mathrm{I}).\label{noflow}
\end{equation}
Hereafter,  \(\mathbf{n}\) denotes the outward unit normal to \(\partial\Omega\).

On the inlet and outlet boundaries \(\Gamma_\mathrm{in}\cup\Gamma_\mathrm{out}\),
 the velocity, the partial densities and the temperature are specified.
Due to the characteristics of the  domain, the inlet velocity is constantly specified on the \(y\) direction.
\begin{itemize}
\item for a.e. \((x,0,z)\in \Gamma_\mathrm{in}\):
\begin{align*}
\mathbf{u}(x,0,z) &= u_\mathrm{in}\mathbf{e}_y\equiv (0,u_\mathrm{in},0) ;\\
\rho_i(x,0,z) &=\rho_{i,\mathrm{in}};\\
 \theta(x,0,z) &=\theta_\mathrm{in}. 
\end{align*}
\item for a.e. \((x,L,z)\in \Gamma_\mathrm{out}\):
\begin{align*}
\mathbf{u}(x,L,z) &= \mathbf{u}_\mathrm{out};\\
\rho_i(x,L,z) &=\rho_{i,\mathrm{out}};\\
 \theta(x,L,z)& =\theta_\mathrm{out}. 
\end{align*}
\end{itemize} 
We refer to \cite{pemfc}, in where the homogeneous Dirichlet condition is assumed, whenever
 the general case for prescribed partial densities and temperature can be handled by subtracting  background
profile that fits the specified  functions.

On the current collector wall boundary \(\Gamma_\mathrm{cc}\),
 the electric potential is prescribed through the cell voltage
\(E_\mathrm{cell}=\phi|_{\Gamma_\mathrm{cc,c}} - \phi|_{\Gamma_\mathrm{cc,a}} \), that means
\begin{equation}\label{Ecell}
\phi=E_\mathrm{cell} \mbox{ on } \Gamma_\mathrm{cc,c}\quad\mbox{and}\quad \phi=0
\mbox{ on } \Gamma_\mathrm{cc,a}.
\end{equation}
This reflects the movement of the electrons in the GDLs, namely  \(\Omega_\mathrm{a}\cup \Omega_\mathrm{c}\),  in the negative \(x\) direction.
Although the fuel reactions release approximately 1.5 joules per coulomb of electronic charge transferred and thus can be assigned a potential of 
\SI{1.5}{\volt} \cite{wetton},
\(E_\mathrm{cell}\) is known to around \SI{0.9}{\volt} \cite{sheng}.

On the remaining  wall boundary \(\Gamma_\mathrm{w}\setminus\Gamma_\mathrm{cc}\),
 the no outflow \( \mathbf{j} \cdot\mathbf{n}=0\) is considered.

Finally, the Newton law of cooling, which is
mathematically known as  the Robin-type boundary condition, is considered
\begin{equation}\label{newton}
 \mathbf{q}\cdot\mathbf{n} = h_c(\theta)(\theta-\theta_e) \mbox{ on } \Gamma_\mathrm{w} ,
\end{equation}
where \(h_c\) denotes the conductive heat transfer coefficient, which may 
 depend both on the spatial variable and the temperature function \(\theta\),
 and  \(\theta_e\) denotes the external coolant stream temperature at the wall.

\subsection{On the fluid-porous  interface  \(\Gamma\)}
\label{porousfluidbd}

The unit outward  normal \(\mathbf{n}\) to the interface boundary \(\Gamma\) pointing from the fluid region to the porous
medium is \(\mathbf{e}_x\) on int\(\left(\partial\Omega_\mathrm{fuel}\cap \partial \Omega_\mathrm{a}\right)\)
and \( - \mathbf{e}_x\) on int\(\left(\partial\Omega_\mathrm{air}\cap \partial \Omega_\mathrm{c}\right)\).

We consider the continuity of mass flux, a  constant interface temperature,
 and the balance of normal Cauchy stress vectors (namely,   \(\sigma_{fN}+\sigma_{pN}=0\))
\begin{align}
\mathbf{u}\cdot\mathbf{e}_x &= \mathbf{u}_\mathrm{D}\cdot\mathbf{e}_x;\label{cmf}\\
 \theta_\mathrm{f} &=\theta_\mathrm{p}; \label{ttfp} \\
(\tau \cdot \mathbf{e}_x)\cdot\mathbf{e}_x &= [p] := p_\mathrm{f}-p_\mathrm{p} ,\label{cns}
\end{align}
where \([\cdot]\) denotes the jump of a quantity across the interface in direction to the fluid medium.
 The condition \eqref{cmf} guarantees that the exchange of fluid between the two domains is conservative. 

We assume the fluid flow is almost parallel to the interface
 and the Darcy velocity is much smaller than the slip velocity. Thus,
 the Beavers--Joseph--Saffman (BJS)  interface boundary condition may be considered \cite{gurau}
\begin{equation}
(\tau \cdot\mathbf{n})\cdot \mathbf{e}_j =-\beta \mathbf{u} \cdot\mathbf{e}_j\qquad (j=y,z) \label{bj}
\end{equation}
where the coefficient \(\beta=\alpha_{BJ} K^{-1/2}>0\) denotes the Beavers--Joseph slip coefficient,
 with \(\alpha_{BJ}\) being dimensionless and characterizing the nature of the porous surface.

The heat transfer transmission  is  completed by  the continuous heat flux condition
\begin{equation}
\label{bcsp}
\mathbf{q}_\mathrm{f} \cdot\mathbf{e}_x= -\mathbf{q}_\mathrm{p}\cdot\mathbf{e}_x .
  \end{equation}

Finally,  the potential is assumed to be neglected
\begin{equation}\label{phi0}
\phi=0\mbox{ on } \Gamma \cap\partial\Omega_\mathrm{a},
\end{equation}
while on \(\Gamma \cap\partial\Omega_\mathrm{c}\) it is simply assumed \(\nabla\phi\cdot\mathbf{n} = - \partial_x\phi =0\).

\subsection{On the  membrane interface \(\Gamma_\mathrm{CL}=\Gamma_\mathrm{a}\cup\Gamma_\mathrm{c}\)}
\label{interface}

The unit outward  normal  \(\mathbf{n}\) to the interface boundary \(\Gamma_\mathrm{CL}\) pointing from the backing layers to the proton conducting
membrane is \(\mathbf{e}_x\) on \(\Gamma_\mathrm{a}\)
and \( - \mathbf{e}_x\) on \(\Gamma_\mathrm{c}\).

The overall  balanced cell reactions are
\begin{description}
\item[H\(_2\)PEMFC] 
\(
2\mathrm{H}_2+ \mathrm{O}_2\rightarrow 2\mathrm{H}_2\mathrm{O}\), \(  E^0_\mathrm{cell}=\SI{1.23}{\volt};
\)

\item[DMFC]
\(
2\mathrm{CH}_4\mathrm{O}+3\mathrm{O}_2\rightarrow 2\mathrm{CO}_2+4\mathrm{H}_2 \mathrm{O},\)  \( E^0_\mathrm{cell}=  \SI{1.21}{\volt} ,
\)
\end{description}
 which are the
sum of two electrochemical reactions (so called half cell reactions) that occur at the electrodes.

On \(\Gamma_\mathrm{a}=\partial\Omega_\mathrm{a}\cap\overline{\Omega}_\mathrm{m}\),
 it occurs the oxidation reaction of the fuel, that is,
 \[  
 \mathbf{j}_1\cdot\mathbf{e}_x= -\frac{s_1M(\mathrm{fuel})}{nF} j_{a} \quad \mbox{ a.e. on }\Gamma_\mathrm{a} ,
 \]  
where \(n\) stands for the number of electrons that participate in the half cell reaction and  \(s_1\) is  the anodic stoichiometry number.
 
On \(\Gamma_\mathrm{c}=\partial\Omega_\mathrm{c}\cap\overline{\Omega}_\mathrm{m}\),
 it occurs the oxygen reduction reaction, that is,
 \[  
 \mathbf{j}_1\cdot\mathbf{e}_x= -\frac{s_1M(\mathrm{O}_2)}{nF} j_{c} \quad \mbox{ a.e. on }\Gamma_\mathrm{c} ,
 \]  
 with  the cathodic stoichiometry number \(s_1\).

Thus, the electric current may be modeled by 
\begin{equation}\label{BVeq}
\mathbf{j}\cdot\mathbf{n} =j_\ell(\eta_\ell) \quad \mbox{ a.e. on }\Gamma_\ell  ,\quad ( \ell = \mathrm{a,c}) ,
\end{equation}
where the reaction rates \(j_\ell\) [\si{\ampere\per\square\metre}] are given by
\begin{equation}\label{defeta}
j_\ell (\eta)=\left\lbrace
\begin{array}{ll}j_{\ell,L}  \frac{2 j_{\ell,0}  \sinh [ \eta /B_\ell]}{ j_{\ell,L} +2  j_{\ell,0}  \sinh[ \eta/B_\ell]} &\quad\mbox{ if }\eta\geq 0\\
-j_\ell(-\eta) &\quad\mbox{ if } \eta<0
\end{array}\right.
\end{equation}
with \(B_\ell = R\theta_\ell/F\) being the Tafel slope at  \(\ell= a, c\),
for  some reference temperatures \(\theta_a\) and \(\theta_c\). Here, it is considered that
 \(\eta_\ell =\phi_\ell-\phi_m-\phi_r\) stands for the overpotential (\(\ell = a, c\)), for some reference potential \(\phi_r\),
the limiting current \(j_{\ell,L}\),  and  some
  \(j_{\ell,0}>0\) only spatial dependent being such that \(j_{\mathrm{a},0} > >j_{\mathrm{c},0}\)  (\(j_{\mathrm{a},0} =\SI{1800}{\ampere\per\metre\squared}\)
and \( j_{\mathrm{c},0} =\SI{0.0132}{\ampere\per\metre\squared}\)  \cite{suhpark}).

We emphasize that the experimental potential  jumps at the interface, \textit{i.e.} the anodic and the cathodic overpotentials are, respectively,
 \(\eta_a < 0 \) and \(\eta_c > 0\).
The modeling \eqref{BVeq}-\eqref{defeta} avoids the existence of infinitely many non-trivial solutions that happens on the Steklov problem
\cite{pagani}.

For the discussion of the Butler--Volmer and Bernardi--Verbrugge boundary conditions, we may refer to \cite{dick}.

\section{Variational formulation and main result}
\label{framework}

In the framework of Sobolev and Lebesgue functional spaces, for \(r>1\), we introduce the following spaces of test functions
\begin{align*}
\mathbf{V}(\Omega_f) =& \{ \mathbf{v}\in \mathbf{H}^1(\Omega_\mathrm{f}):\ 
 \mathbf{v}=\mathbf{0}\mbox{ on }\Gamma_\mathrm{in}\cup\Gamma_\mathrm{out}; \
 \mathbf{v}\cdot\mathbf{n}=0\mbox{ on }\Gamma_\mathrm{w} \};\\
H(\Omega_p) =& \{v\in L^{2}(\Omega_\mathrm{p}):\  v_a :=v|_{\Omega_\mathrm{a}}\in H^{1}(\Omega_\mathrm{a}) , \ 
 v_c :=v|_{\Omega_\mathrm{c}}\in H^{1}(\Omega_\mathrm{c}) ,\\
&\   v_m :=v|_{\Omega_\mathrm{m}}\in H^{1}(\Omega_\mathrm{m}) ,
\ v_a =v_m \mbox{ on }\Gamma_\mathrm{a}, \  v_c =v_m \mbox{ on }\Gamma_\mathrm{c}\} ;\\
V(\Omega) =& \{v\in L^{2}(\Omega ):\  v_\ell:=v|_{\Omega_\ell}\in H^{1}(\Omega_\ell) , \ \ell=\mathrm{a},\mathrm{c}, \mathrm{f},
\   v_m :=v|_{\Omega_\mathrm{m}}\in H^1(\Omega_\mathrm{m}), \\
&\ v_\ell =v_m \mbox{ on }\Gamma, \ \ell=\mathrm{a}, \mathrm{c},  \  v=0\mbox{ on }\Gamma_\mathrm{in}\cup\Gamma_\mathrm{out} \};\\
V(\Omega) =& \{ v\in H^{1}(\Omega):\ v=0\mbox{ on }\Gamma_\mathrm{in}\cup\Gamma_\mathrm{out} \};\\
V_r(\Omega_p) =& \{ v\in W^{1,r}(\Omega_\mathrm{p}):\ v=0\mbox{ on }\Gamma_\mathrm{cc}  \cup (\Gamma  \cap\partial\Omega_\mathrm{a}) \};\\
H(\Omega) =& \{v\in H^{1}(\Omega):\ v_f:=v|_{\Omega_\mathrm{f}},\  v_p:=v|_{\Omega_\mathrm{p}},\
 v_f=v_p \mbox{ on }\Gamma\} , 
\end{align*}
with their usual norms. 
Considering that the Poincar\'e inequality occurs whenever the trace of the function vanishes on a part with positive measure of the boundary \(\partial\Omega\),
then the  Hilbert spaces, \( \mathbf{V}(\Omega_f)\), \( V_2(\Omega_p) \) and  \(V(\Omega) \),  are  endowed with the standard seminorms (cf. Section \ref{tdt}). 
We denote \(V(\Omega_p)=V_2(\Omega_p)\), for the sake of simplicity.

The fuel cell problem, which its strong formulation is  stated in Section \ref{spemfc}, is equivalent to the following variational formulation.
\begin{definition} \label{dwt}
We say that the function \((\mathbf{u},p,\bm{\rho},\theta,\phi)\) is a weak solution to the fuel cell problem,
  if it satisfies the following variational formulations to
  \begin{itemize} 
\item  the momentum conservation (Beavers--Joseph--Saffman/Stokes--Darcy problem)
\begin{align} \label{motionwbj} 
\int_{\Omega_\mathrm{f}}\mu(\theta)D\mathbf{u}:D\mathbf{v}  \dif{x}+
\int_{\Omega_\mathrm{f}}\lambda(\theta)\nabla\cdot\mathbf{u}\nabla\cdot\mathbf{v}  \dif{x} 
+ \int_{\Omega_\mathrm{p}}  \frac{K_g(p)}{\mu(\theta)}\nabla p\cdot\nabla v\dif{x} \nonumber \\
+\int_{\Gamma} \beta(\theta)\mathbf{u}_T\cdot\mathbf{v} _T\dif{s}+\int_{\Gamma }  p\mathbf{v}\cdot\mathbf{n} \dif{s}
-\int_{\Gamma }  \mathbf{u}\cdot\mathbf{n}v \dif{s}\nonumber \\
=R_M\int_{\Omega_\mathrm{f}} \rho \theta \nabla\cdot\mathbf{v} \dif{x}, 
\end{align}
 holds for all \((\mathbf{v},v)\in  \mathbf{V}(\Omega_f)\times H(\Omega_p)\).
Here,  \(\rho = \rho_1+\rho_2\).

\item the species conservation (\(i=1,2\))
\begin{align}
\int_{\Omega_\mathrm{f}}  \rho_1 \mathbf{u}\cdot\nabla  v\dif{x} + 
\int_{\Omega} D_1(\theta )\nabla \rho_1 \cdot\nabla v\dif{x} 
 +\frac{F}{R}\int_{\Omega_\mathrm{m}} \psi( \rho_1) \frac{D_{1} (\theta) }{ \theta } \nabla\phi \cdot\nabla v\dif{x}  \nonumber \\
+ \int_{\Omega_\mathrm{m}} D_{12}( \theta )\nabla \rho_2 \cdot\nabla v\dif{x} + 
\int_{\Omega} \rho_1 S_{1}(\rho_1, \theta )\nabla \theta\cdot\nabla v \dif{x} = 0;   \label{wvf1} \\
\int_{\Omega_\mathrm{f}}  \rho_2 \mathbf{u}\cdot\nabla  v\dif{x} + 
\int_{\Omega} D_2(\theta )\nabla \rho_2 \cdot\nabla v\dif{x} 
+ \int_{\Omega_\mathrm{m}} n_\mathrm{d}( \theta )\nabla \rho_1 \cdot\nabla v\dif{x}  \nonumber \\ + 
\int_{\Omega} \rho_2S_{2}(\rho_2,\theta )\nabla \theta\cdot\nabla v \dif{x}  = 0,   \label{wvf2}
\end{align}
 holds for all  \(v\in V(\Omega)\). Here, we set
\[
\psi(z)=\left\{\begin{array}{ll}
z&\mbox{ if } 0\leq z\leq \rho_{1,m}\\
0&\mbox{ otherwise}
\end{array}\right. 
\]
for some \(\rho_{1,m}>0\).
 
\item the energy conservation
\begin{align}\label{wvfi1}  
\int_{\Omega}  k(\theta)\nabla\theta \cdot\nabla v\dif{x} 
+ \int_{\Gamma_\mathrm{w}} h_c(\theta)\theta v \dif{s}  \nonumber \\
+\sum_{j=1}^{2 } \frac{R}{M_j}
\int_{\Omega} \theta^2 D_{j}'( \rho_j,\theta )\nabla \rho_j \cdot\nabla v\dif{x} +
\int_{\Omega_\mathrm{m}} \Pi(\theta) \sigma_m(\rho_2,\theta )\nabla\phi \cdot\nabla v\dif{x} \nonumber \\
 = \int_{\Gamma_\mathrm{w}}  h_c(\theta)\theta_e  v \dif{s} 
+\int_{\Omega_\mathrm{a}\cup\Omega_\mathrm{c}} \sigma( \bm{\rho},\theta ) |\nabla\phi|^2v\dif{x} ,
\end{align}
 holds for all  \(v\in  V(\Omega)\).
 
\item  the electricity conservation  
\begin{align}\label{wvfphi}
\int_{\Omega_\mathrm{p}}\sigma ( \bm{\rho},\theta )\nabla\phi\cdot\nabla w\dif{x}
 + \frac{1}{M_1} \int_{\Omega_\mathrm{m}}  \kappa(\theta )\nabla \rho_1 \cdot \nabla w \dif{x}  \nonumber \\ 
+ \int_{\Omega_\mathrm{m}} \alpha_\mathrm{S}(\theta) \sigma_m( \rho_2, \theta )  \nabla \theta \cdot \nabla w \dif{x} \nonumber\\ 
+\int_{\Gamma_\mathrm{a}} j_a(\phi_a-\phi_m) (w_a-w_m)\dif{s}
+  \int_{\Gamma_\mathrm{c}}j_c(\phi_c- \phi_m-E_\mathrm{cell}) (w_c-w_m) \dif{s} =0,
\end{align}
holds for all  \(w\in V(\Omega_p) \). 
 \end{itemize}
\end{definition}
Hereafter,  the notation \(\dif{x}\) refers to the 2D \(\dif{x}\dif{y}\) and the 3D \(\dif{x}\dif{y}\dif{z}\)
and whenever this may be misunderstanding we use  \(\dif{x_1}\dif{x_2}\dif{x_3}\).
We use the notation \(\dif{s}\) for the surface element in
the integrals on the boundary  as well as any subpart of the  boundary \(\partial\Omega\).
In \eqref{wvfphi},  the subscripts denote the restriction to \(\Omega_\ell\), \(\ell=\) a, c, or \(\Omega_\mathrm{m}\).

\begin{remark}\label{rpsi} The truncation \(\psi\) is assumed,
which is consistent with the real behavior of the partial density of H\(_3\)O\(^+\) in the membrane.
Mathematically speaking, it avoids some extra regularity of the weak solutions and, consequently,
the even more restriction on the smallness conditions.
We  emphasize that the \(L^\infty \)-bound of solutions of elliptic equations is not straightforward true for elliptic systems \cite{staicu}.
\end{remark}

The set of hypothesis is as follows.
\begin{description}
\item[(H1)] The viscosities \(\mu\) and \(\lambda\) are  assumed to be Carath\'eodory functions from \(\Omega_\mathrm{f}\times\mathbb{R}\) into \(\mathbb{R}\),
\textit{ i.e.} measurable with respect to space variable and
  continuous with respect to other variable,  such that
\begin{align}
\label{mu}
\exists \mu_\# ,  \mu^\# >0: &\  \mu_\# \leq\ \mu(x,e) \leq \mu^\# ; \\
\exists \lambda^\#>0: &\  - \mu/n\leq\lambda(x,e) \leq  \lambda^\# ,
\label{nu3}
\end{align}
for a.e. \(x\in\Omega_\mathrm{f}\) and  for all \( e\in\mathbb{R}\). While \(K_g\) 
is assumed to be Carath\'eodory function from \(\Omega_\mathrm{p}\times\mathbb{R}\) into \(\mathbb{R}\) such that
\begin{equation}\label{defkl}
\exists K_l,K_l^\#>0 : \ K_l\leq K_g(x,e)\leq K_l^\#,
\end{equation}
for a.e. \(x\in\Omega_\mathrm{p}\) and  for all \( e\in\mathbb{R}\).

\item[(H2)]  
The leading coefficients \(D_i\) and \(k\)  are Carath\'eodory  functions from  \(\Omega\times\mathbb{R}\) to \(\mathbb{R}\)  and
  \(\sigma\)  is a   Carath\'eodory  function from  \(\Omega\times\mathbb{R}^{3}\) to \(\mathbb{R}\) such that
 \(\sigma(x,\cdot)\equiv \sigma_m\in C(\mathbb{R}^2)\)  for a.e \(x\in\Omega_\mathrm{m}\). 
Moreover, they satisfy 
\begin{align}
\exists  D_i^\#, D_{i,\#} >0:\ & D_{i,\#} \leq  D_i(x,e)\leq D_i^\#,\quad \mbox{ for a.e. }  x\in \Omega_\mathrm{m} \cup
\Omega_\mathrm{GDL}; \label{Dif}\\
\exists D_{i,m}^\#,D_{i,m} >0:\ &  D_{1,_m} \leq  D_1(x,e)\leq 
\left\{\begin{array}{ll}
D_{1,m}^\#|e|/T^\#& \mbox{if } |e|\leq T^\#\\
D_{1,m}^\# &\mbox{if } |e|>T^\#
\end{array}\right.\\
&  D_{2,_m} \leq  D_2(x,e)\leq D_{2,m}^\#, 
\quad \mbox{ for a.e. }  x\in \Omega_\mathrm{m} ; \label{Dip}\\
\exists k^\#, k_\# >0:\ & k_\#\leq k(x,e)\leq k^\#, \quad\mbox{ for a.e. } x\in\Omega; \label{km}\\
\exists\sigma^\#, \sigma_\#>0:\ &\sigma_\#\leq \sigma(x,\mathbf{e})\leq \sigma^\#, \quad \mbox{ for a.e. }  x\in \Omega_\mathrm{a}\cup  \Omega_\mathrm{c};\label{smp}\\ 
\exists\sigma_m^\#, \sigma_{m,\#}>0: \ &  \sigma_{m,\#}\leq \sigma_m(d,e) \leq  \sigma_m^\#, 
\label{smm}
\end{align}
 for all  \(d,e\in\mathbb{R}\) and \(\mathbf{e}\in\mathbb{R}^{3}\).  

\item[(H3)]  The cross-effect  Peltier, Seebeck, Soret, Dufour and binary diffusion 
 coefficients \(\Pi,\alpha_\mathrm{S},S_i\), \(D_i', D_{ij}\) (\(i,j=1, 2\) with \(j\not=i\)) are
  Carath\'eodory functions such that
   \begin{align}
\exists\Pi^\#>0:\ & |\Pi(x,e)|\leq \Pi^\# , \quad \mbox{ for a.e. }  x\in \Omega_\mathrm{p} ;\label{pmax}\\
\exists\alpha^\#>0:\ & |\alpha_\mathrm{S}(x,e)|\leq \alpha^\#, \quad \mbox{ for a.e. }  x\in \Omega_\mathrm{m} ;\label{amax}\\
\exists S_i^\#>0:\ & |dS_i(x,d,e)|\leq S_i^\#, \quad \mbox{ for a.e. }  x\in \Omega ;\label{dsmax}\\
\exists (D_i')^\#>0:\ & (R/M_i) e^2|D_i'(x,d,e)|\leq (D_i')^\#, \quad \mbox{ for a.e. }  x\in \Omega ;\label{edmax}\\
\exists  D_{ij}^\#>0:\ & |  D_{ij}(x,e)| \leq D_{ij}^\#,\quad \mbox{ for a.e. }  x\in \Omega , \label{Dij}
  \end{align}
 for all \(d,e\in \mathbb R\).
Moreover, we assume
\begin{align}\label{defai}
a_{1,\#}& := \frac{1-\epsilon_1-\epsilon_2 -\epsilon_3 }{2} D_{1,\#} - \frac{1}{\epsilon_6} \frac{((D'_1)^\#)^2}{k_\#} >0; \\
a_{1,m} & := \frac{1-\epsilon_1-\epsilon_3}{2}  D_{1,m} -\frac{1}{\epsilon_6} \frac{ ((D'_1)^\#)^2}{k_\#} 
 - \frac{1}{\epsilon_4} \frac{(D^\#_{21})^2}{D_{2,m} }  - \frac{F^2}{\epsilon_8 M_1^2}\frac{ (D_{1,m}^\#)^2}{\sigma_{m,\#}}  >0; \\
a_{2,\#}& := \frac{1-\epsilon_4-\epsilon_5}{2}D_{2,\#}  - \frac{1}{\epsilon_6} \frac{((D'_2)^\#)^2}{k_\#} >0; \\
a_{2,m}& := \frac{1-\epsilon_4}{2} D_{2,m} - \frac{1}{\epsilon_6}\frac{ ((D'_2)^\#)^2}{k_\#}  - \frac{1}{\epsilon_1}  \frac{(D^\#_{12})^2}{D_{1,m} } >0; \\
a_{3,\#}& := \frac{ 1-\epsilon_6-\epsilon_7}{2} k_\# -  \frac{1}{\epsilon_2}   \frac{(S_1^\#)^2}{D_{1,\#}}
 -  \frac{1}{\epsilon_5}   \frac{(S_2^\#)^2}{D_{2,\#}}  >0; \\
a_{3,m}& := \frac{1-\epsilon_7}{2} k_\# -  \frac{1}{\epsilon_2}  \frac{(S_1^\#)^2}{D_{1,m}} -  \frac{1}{\epsilon_5}   \frac{(S_2^\#)^2}{D_{2,m}} 
- \frac{1}{\epsilon_9 }  (\alpha ^\#)^2\sigma^\#  >0; \\
\label{defa4}
a_{4,m}& := \sigma_{m, \# }\left( 1- \frac{\epsilon_8+\epsilon_9}{2 } 
- \frac{1}{2\epsilon_7} \frac{(\Pi^\#)^2\sigma_m^\#}{ k_{\#}}\right) - \frac{1}{2\epsilon_3} \frac{ (\rho_{1,m}\kappa^\#)^2}{D_{1,m}} >0 .
\end{align}
for some \(\epsilon_1,\cdots,\epsilon_9>0\) being such that \(\epsilon_1+\epsilon_2+\epsilon_3<1\), \(\epsilon_4+\epsilon_5<1\),
\(\epsilon_6 +\epsilon_7<1\) and \(\epsilon_8 + \epsilon_9 <2\).
Here, \(\rho_{1,m}\)  stands for the upper bound given at (H8) and \(\kappa^\#:= F D_{1,m}^\#/(RT^\#)\). 

\item[(H4)] The boundary coefficient \(\beta\)  is assumed to be a Carath\'eodory function from \(\Gamma\times\mathbb{R}\) into \(\mathbb{R}\). 
 Moreover, there exist \(\beta_\#,\beta^\#>0\) such that 
\begin{equation}\label{gamm}
\beta_\#\leq \beta(\cdot,e) \leq \beta^\# ,
\end{equation}
 a.e. in \(\Gamma\), and  for all \(e\in\mathbb{R}\). 
 
\item[(H5)] The boundary coefficient \(h_c\)  is assumed to be a Carath\'eodory function from \(\Gamma_\mathrm{w}\times\mathbb{R}\) into \(\mathbb{R}\). 
 Moreover, there exist \(h_\#,h^\#>0\) such that 
\begin{equation}\label{hmm}
h_\#\leq h_c(\cdot,e) \leq h^\# ,
\end{equation}
 a.e. in \(\Gamma_\mathrm{w}\), and  for all \(e\in\mathbb{R}\).
 
\item[(H6)] The boundary functions \(j_\ell\), \(\ell=\) a, c, are assumed to be the increasing, odd continuous functions from \(\mathbb{R}\) into \(\mathbb{R}\), 
defined in \eqref{defeta}.

\item[(H7)] There exists \( \mathbf{u}_0\in \mathbf{H}^{1}(\Omega_\mathrm{f})\) such that \(\mathbf{u}_0=u_\mathrm{in}\mathbf{e}_y\) on \(\Gamma_\mathrm{in} \),
\(\mathbf{u}_0=\mathbf{u}_\mathrm{out}\) on \(\Gamma_\mathrm{out} \) and
\(\mathbf{u}_0=\mathbf{0}\) on \(\partial\Omega_\mathrm{f}\setminus\left(\Gamma_\mathrm{in} \cup \Gamma_\mathrm{out}\right)\). 

\item[(H8)] There exist \(\rho_{1,0}\) and \(\rho_{2,0}\) belong to \(C(\overline\Omega)\) such that  \(\rho_{i,0}=\rho_{i,\mathrm{in}}\)on \(\Gamma_\mathrm{in} \),
  \(\rho_{i,0}=\rho_{i,\mathrm{out}}\)on \(\Gamma_\mathrm{out} \)  and
\(\nabla \rho_{i,0}\cdot\mathbf{n}=0\)  on \(\partial\Omega\setminus\left(\Gamma_\mathrm{in} \cup \Gamma_\mathrm{out}\right)\),
 for \(i=1,2\). Moreover, the lower and  upper bounds \(0\leq \rho_{1,0}\leq \rho_{1,m}\) occur a.e. in \(\Omega_\mathrm{m}\).

\item[(H9)] There exists \(\theta_0 \in H^1(\Omega)\) such that \(\theta_0= \theta_{\mathrm{in}}\)on \(\Gamma_\mathrm{in} \) and
 \(\theta_0= \theta_{\mathrm{out}}\)on \(\Gamma_\mathrm{out} \).
\end{description}
\begin{remark}
The nonstandard assumptions \eqref{defai}-\eqref{defa4} are required for the Legendre--Hadamard ellipticity condition.
\end{remark}

Using the fixed point argument, we establish the following 2D result under the smallness on the data.
\begin{theorem}\label{tmain}
Let \(\Omega\) be a bounded multiregion domain   of \(\mathbb{R}^n\), \(n=2\).
Under the assumptions (H1)-(H9),
the fuel cell problem admits, at least, one solution according to Definition \ref{dwt} such that
\begin{itemize}
\item the velocity \(\mathbf{u}\in \mathbf{u}_0+\mathbf{V}(\Omega_f)\); 
\item the  pressure \(p\in H(\Omega_p)\);
\item the partial densities \( \bm{\rho} \in  \bm{\rho} _0 + [V(\Omega)]^2\);
\item the temperature \(\theta\in  \theta_0 + V(\Omega)\);
\item the potential \(\phi \in  E_{cell}\chi_{\Omega_\mathrm{c}}+  V_r(\Omega_p) \), for \(r>2\),
\end{itemize}
if provided by  the smallness condition
\begin{equation}
root_2 >  \sqrt{\frac{2C_K}{\mu_\#} } C_0 
 +\sqrt{C_KL}   \frac{R_M}{\mu_\#}
\frac{1}{a_{\#} } \left(h^\#\|\theta_e\|_{2,\Gamma_\mathrm{w}}^2 +2\mathcal{B}_0\right) ,\label{small}
\end{equation}
where \(root_2\) is the positive root of the quadratic polynomial \eqref{2polynomial}, \(C_K >1\) is the Korn constant in \eqref{korn}
and  \(C_0, \mathcal{B}_0, a_\#>0\) are constants defined in \eqref{C0},  \eqref{defb0} and \eqref{defaa}, respectively.
\end{theorem}

In the sequel, we focus on the H\(_2\)PEMFC.
Under the operating parameters \cite{gosh,li2011},  the following data are known.
\begin{enumerate}
\item
The viscosities are known  decreasing functions on temperature,   for instance   in the operating temperature range
\( \SIrange{320}{390}{\kelvin}\),
\(\mu_\#\approx \SI{4.2e-5}{\pascal\second}\)  and
\(\mu^\#\approx \SI{4.8e-5}{\pascal\second}\)  for the  air. 
The water viscosity \(\mu_{water} \approx 100 \mu_{air}\).
\item For values of \(K= \SI{1.76E-11}{\square\metre} \) we have \(\beta \) of order \(10^5\).
We may assume \(\beta_\#=1\) in  (H4).

\item The thermal conductivity for H\(_2\) varies from \(k_{\mathrm{H}_2} (\SI{300}{\kelvin}) = \SI{0.18}{\watt\per\metre\per\kelvin}\) 
to  \(k_{\mathrm{H}_2} (\SI{400}{\kelvin}) = \SI{0.2}{\watt\per\metre\per\kelvin}\).
Typical values of the  thermal conductivity  \(k \approx 0.03\),  0.023  and \SI{0.67}{\watt\per\metre\per\kelvin} 
are known for air, for water vapor and for liquid water, respectively.
In \(\Omega_\mathrm{p}\), the thermal conductivity varies in the range \(0.2-0.5\) \si{\watt\per\metre\per\kelvin}.
Considering the electrical conductivity \(\sigma^\# = \SI{120}{\siemens\per\metre}\)  and Peltier coefficients with its maximum of
\(\Pi^\#=0.3\) \cite{suhpark}, then the smallness condition \eqref{defa4} is by validated  by \eqref{kelvin}.
 For the anode, we have 
\(\alpha^\# =0.3/320 << (k_\# /\sigma^\#)^{1/2} \approx 0.04\).

\item For  air/O\(_2\), the inlet velocity is  \(u_\mathrm{in}= \SI{0.2}{\metre\per\second}\) then \(1.9\lesssim \sqrt{\beta^\#}u_\mathrm{in}|\Gamma|^{1/2}\lesssim   6.3 \),
with \(0.03\leq \sqrt{L}\leq 10^{-1} < 2/(1+2\sqrt{2})\approx 0.5\).

\item The inlet concentration of hydrogen is known as around \SI{54}{\mol\per \cubic\metre}.
Recalling \eqref{defrhoi}, we know \(\rho_{1,\mathrm{in}}(\mathrm{H}_2) = \SI{0.1}{\kilogram\per \cubic\metre} >  \rho_{1,m}\).

\item Typical values for diffusion coefficients are \(D_{\mathrm{O}_2} (\theta)=  1.77\times 10^{-4}(\theta/273)^{1.8}\)
 and \(D_{\mathrm{H}_2\mathrm{O}} (\theta)= 2.56\times 10^{-5}(\theta/307)^{2.3}\),
while for hydrogen in water vapor \(D_1\approx \SI{9.15E-5}{\square\metre\per\second}\).
In the membrane \(\Omega_\mathrm{m}\),
 the binary diffusion coefficient \(D_{21} =\SI{1.1E-8}{\square\metre\per\second}\),
while typical values of diffusion coefficients are \(D_{\mathrm{O}_2} (\theta)=  2.88\times 10^{-10}\exp\left[2933(1/313-1/\theta)\right]\),
\(D_{\mathrm{O}_2} (\theta)=  4.1\times 10^{-7}\exp\left[ -2602/\theta \right]\) and \(D_2=0\) \cite{li2011}.

\item Heat transfer coefficients are \(h_{\mathrm{H}_2\mathrm{O}} = \SI{2672}{\watt\per\square\metre\per\kelvin}\),
 \(h_{\mathrm{H}_2} = \SI{824}{\watt\per\square\metre\per\kelvin}\)  and
 \(h_{air} = \SI{1200}{\watt\per\square\metre\per\kelvin}\) \cite{ramousse}.

\end{enumerate}

The method applied in determining explicit constants, namely in Proposition \ref{proppoincare},
does not work in 3D unless the additional assumption of the functions vanish at least on the solid wall basis (\(z=0\)).
However, neither the fluid velocity field nor the partial densities verify the Dirichlet condition on real situations.
The fluid velocity field is only known impenetrable on the solid boundary
while the partial densities exist satisfying  \eqref{noflow}.

\section{Strategy}
\label{strat}
 
Set the \((\bm{\rho},\theta)\)-dependent  \(4\times 4\)-matrix
\[  
\mathsf{A}(\bm{\rho},\theta) =\left[\begin{array}{cccc}
D_1(\theta) & D_{ 12}(\rho_2,\theta)  & \rho_1 S_{1}(\rho_1,\theta) &\rho_1\kappa(\theta)/(R\theta)\\
D_{21}(\rho_1,\theta)  &D_2(\theta) &\rho_ 2 S_2(\rho_2,\theta) & 0 \\
R\theta^2 D'_{1}(\rho_1,\theta) /M_1  & R\theta^2 D_{2}(\rho_2,\theta)/M_2 & k(\theta)  & \Pi(\theta) \sigma_m(\theta) \\
\kappa(\theta)/M_1 & 0  & \alpha_S(\theta) \sigma_m(\theta) & \sigma(\bm{\rho},\theta)
\end{array}\right] .
\]

The existence of the weak solution to the fuel cell problem relies on the fixed point argument
\begin{align}\label{fpa}
\mathcal{T}: (\pi,  \bm{\upsilon} , \Phi) & \in E:=
(H(\Omega_p)/\mathbb{R}) \times [V(\Omega)]^{3} \times L^t(\Omega_\mathrm{a}\cup\Omega_\mathrm{c})\nonumber\\
&\mapsto    (\mathbf{U},p) \in \mathbf{V}(\Omega_f)\times ( H(\Omega_p)/\mathbb{R}) \nonumber \\
& \mapsto (\bm{\Upsilon}, \Theta,\phi_{cc}) \in   [V(\Omega)]^{3} \times V_r(\Omega_p) \nonumber\\
&\mapsto (p, \bm{\Upsilon},\Theta, |\nabla\phi|_{\Omega_\mathrm{a}\cup\Omega_\mathrm{c}}|^2)
\end{align}
where
\begin{itemize}
\item \((\mathbf{U},p) =(\mathbf{u} ,p) (\pi, \bm{\varrho},\xi)  \) stands for the auxiliary velocity-pressure pair solving the 
homogeneous Dirichlet--BJS/Stokes--Darcy problem
\begin{align} 
\int_{\Omega_\mathrm{f}}\mu(\xi)D\mathbf{U}:D\mathbf{v}  \dif{x}+
\int_{\Omega_\mathrm{f}}\lambda(\xi)\nabla\cdot\mathbf{U}\nabla\cdot\mathbf{v}  \dif{x}\nonumber \\
+\int_{\Gamma} \beta(\xi)\mathbf{U}_T\cdot\mathbf{v}_T \dif{s}+
 \int_{\Omega_\mathrm{p}}\frac{K_g(\pi)}{\mu (\xi)}\nabla p\cdot\nabla v\dif{x}
+ \int_\Gamma p\mathbf{v}\cdot \mathbf{n}\dif{s} -
 \int_{\Gamma} \mathbf{U}\cdot\mathbf{n} v\dif{s}\nonumber \\
= R_M \int_{\Omega_\mathrm{f}} \varrho\xi \nabla\cdot\mathbf{v} \dif{x}
-G(\xi, \mathbf{u}_0, \mathbf{v} ,v)
,\ \forall (\mathbf{v},v) \in \mathbf{V}(\Omega_f)\times H(\Omega_p),\label{wvfup}
\end{align}
where
\begin{align}\label{defvarrho}
\varrho = \varrho_1+ \varrho_2 & \quad (\varrho_1 = \upsilon_1+\rho_{1,0} \mbox{ and }  \varrho_2 = \upsilon_2+\rho_{2,0} )\mbox{ and }  \xi =\upsilon_3+\theta_{0} ; \\
G(\xi, \mathbf{z}, \mathbf{v} ,v)& = 
\int_{\Omega_\mathrm{f}}\mu(\xi)D\mathbf{z}:D\mathbf{v}  \dif{x}
+\int_{\Omega_\mathrm{f}}\lambda(\xi)\nabla\cdot\mathbf{z}\nabla\cdot\mathbf{v}  \dif{x}.  \nonumber 
\end{align}
We define \(\mathbf{u}=\mathbf{U}+ \mathbf{u}_0\).

\item \( (\Upsilon_1,\Upsilon_2, \Theta, \phi_{cc}) = (\bm{\rho}, \theta, \phi  )(\mathbf{w}, \bm{\varrho},\xi,\Phi) \) 
stands for the auxiliary partial densities, temperature and potential solving the coupled problem
\begin{align}
\int_{\Omega_\mathrm{f}}  \Upsilon_1 \mathbf{w}\cdot\nabla  v\dif{x} + 
\int_{\Omega} D_1( \xi )\nabla \Upsilon_1 \cdot\nabla v\dif{x}  
+ \int_{\Omega} D_{12}( \varrho_2,\xi )\nabla \rho_2 \cdot\nabla v\dif{x} \nonumber \\
+ 
\int_{\Omega} \varrho_1 S_{1 }( \varrho_1,\xi )\nabla \theta\cdot\nabla v \dif{x}  + \frac{1}{R}
\int_{\Omega_\mathrm{m}} \psi(\varrho_1) \frac{\kappa (\xi )}{\xi} \nabla\phi \cdot\nabla v\dif{x} 
 = - g_1(\mathbf{w},  \xi, \rho_{1,0} ,v); \label{aux-pdi1} 
\end{align}
\begin{align}
\int_{\Omega_\mathrm{f}}  \Upsilon_2 \mathbf{w}\cdot\nabla  v\dif{x} + 
\int_{\Omega} D_2( \xi )\nabla \Upsilon_2 \cdot\nabla v\dif{x}  
+ \int_{\Omega} D_{21}( \varrho_1,\xi )\nabla \rho_1 \cdot\nabla v\dif{x} \nonumber \\
+ 
\int_{\Omega} \varrho_2 S_{2}( \varrho_2,\xi )\nabla \theta\cdot\nabla v \dif{x}  
 = - g_2(\mathbf{w},  \xi, \rho_{2,0} ,v) ; \label{aux-pdi2} 
\end{align}
\begin{align} 
\int_{\Omega}  k(\xi)\nabla\Theta \cdot\nabla v\dif{x} + \int_{\Gamma_\mathrm{w}} h_c(\xi)\Theta v \dif{s}  \nonumber \\
+\sum_{j=1}^{2 } \frac{R}{M_j}
\int_{\Omega} \xi^2 D'_{j}( \varrho_j,\xi)\nabla \rho_j \cdot\nabla v\dif{x} +
\int_{\Omega_\mathrm{m}} \Pi(\xi) \sigma_m(\varrho_2,  \xi)\nabla\phi\cdot\nabla v\dif{x} \nonumber \\
 = \int_{\Gamma_\mathrm{w}}  h_c(\xi)\theta_e  v \dif{s} 
+\int_{\Omega_\mathrm{a}\cup\Omega_\mathrm{c}} \sigma( \bm{\varrho},\xi) \Phi v\dif{x} -g_3(\xi, \theta_0,v) ;\label{aux-tt} 
\end{align}
\begin{align} 
\int_{\Omega_\mathrm{p}}\sigma(\bm{\varrho},\xi) \nabla\phi\cdot\nabla w\dif{x}
 +  \frac{1}{M_1} \int_{\Omega_\mathrm{m}}   \kappa(\xi)\nabla\rho_1 \cdot \nabla w \dif{x} \nonumber\\
+ \int_{\Omega_\mathrm{m}}  \alpha_\mathrm{S}(\xi)  \sigma_m(\varrho_2, \xi) \nabla\theta \cdot \nabla w \dif{x}\nonumber\\
+  \int_{\Gamma_\mathrm{a}} j_{a} (\phi_{cc,a}-\phi_{cc,m}) (w_a-w_m)\dif{s}
+ \int_{\Gamma_\mathrm{c}} j_{c} (\phi_{cc,c}-\phi_{cc,m}) (w_c-w_m)\dif{s}=0,  \label{aux-phicc}\end{align}
 for all  \(v\in V(\Omega)\) and \( w\in V(\Omega_p)\),
with  \(\mathbf{w}=\mathbf{u}(\pi, \bm{\varrho},\xi)\) being  the auxiliary velocity field given at Proposition  \ref{proppxi} and
\begin{align*}
g_i(\mathbf{w}, \xi, z,v) &=\int_{\Omega_\mathrm{f}}  z \mathbf{w}\cdot\nabla  v\dif{x} + 
\int_{\Omega} D_i( \xi )\nabla z \cdot\nabla v\dif{x} ; \\
g_3(\xi,z,v) &= \int_{\Omega}  k(\xi)\nabla z \cdot\nabla v\dif{x} + \int_{\Gamma_\mathrm{w}} h_c(\xi) z v \dif{s}, 
\end{align*}
for \(i=1,2\).
Here,  \(\bm{\rho}=\bm{\Upsilon}+ \bm{\rho}_0\), \(\theta=\Theta+\theta_0\)  and
 \(\phi=\phi_{cc} + E_{cell}\chi_{\Omega_\mathrm{c}} \), 
with \(\chi_{\Omega_\mathrm{c}}\) denoting the characteristic function.
\end{itemize}
The proofs of existence of a unique solution to each one of these systems involves the use of Lax--Milgram and Browder--Minty Theorems, respectively.

\begin{remark}
In the presence of the unbounded function \(\rho_1\in H^1(\Omega)\) in \eqref{aux-pdi1}, we freeze 
the partial density \(\rho_1\) as \(\varrho_1\) and we also take its real behavior by the truncation \(\psi\) (see Remark \ref{rpsi}),
in contrast with the abstract argument used in the work \cite{pemfc}.
\end{remark}

\section{Auxiliary results}
\label{tdt}

 Throughout this section, the space dimension \(n\) is kept general as possible.
To precise the quantitative estimates either the restriction 3D is required in Proposition \ref{proppoincare}
or the restriction of \(n=2\) is  required in Proposition \ref{propn2} and Lemmata \ref{lemae} and \ref{lemaw}.

First, we recall  the Korn  inequality and the Poincar\'e-type inequality known as Deny--Lions lemma
\begin{align}
\|\nabla\mathbf{v}\|_{2,\Omega}^2\leq C_K\|D\mathbf{v}\,\|_{2,\Omega}^2, & \quad\forall \mathbf{v}\in \mathbf{H}^{1}(\Omega); \label{korn}\\
\inf_{\alpha\in \mathbb{R}} \|v-\alpha\|_{2,\Omega_\mathrm{p}} \leq C_{\Omega_\mathrm{p}} \|\nabla v\|_{2,\Omega_\mathrm{p}},&
 \quad \forall v\in H^1(\Omega_\mathrm{p}) .\label{deny}
\end{align}
for some constant  \(C_K>1\),  and  some constant \(C_{\Omega_\mathrm{p}} \) only dependent on the domain \(\Omega_\mathrm{p}\).

Indeed, the \(q\)-Poincar\'e  inequality can have different forms (\(q>1\)),  in particular 
\begin{align}
\|v\|_{q,\Omega} \leq C_{\Omega}\left( \|\nabla v\|_{q,\Omega} +\left| \int_D v\dif{s}\right|\right),  \quad \forall v\in W^{1,q}(\Omega),\label{qpoincare}\\
\|v- (v)_D\|_{q,\Omega} \leq C_{\Omega}  \|\nabla v\|_{q,\Omega}, \quad \forall v\in W^{1,q}(\Omega),\label{bpoincare}
\end{align}
whenever \(\Omega\) is a bounded Lipschitz domain and \(D\) is measurable subset of \(\partial\Omega\) with positive (\(n-1\))-dimensional Lebesgue measure.
The Poincar\'e constant  \(C_{\Omega}\)  in the generalized Friedrichs inequality \eqref{qpoincare}  is not explicitly determined because the proof relies on the
contradiction argument. These abstract constants  are not useful for establishing quantitative estimates.
The Poincar\'e constant  \(C_{\Omega}\)  in \eqref{bpoincare} is known sharp equal to \(C_\Omega= \lambda_1^{-1/2}\), in the quadratic case (\(q=2\)),
where \(\lambda_1\) is the smallest positive eigenvalue of the  mixed Steklov problem.

We refer to \cite{zheng}, some quantitative estimates for the  Friedrichs-type inequalities
\[
\|v\|_{2,\Omega} \leq |\Omega|^{-1/2}\left(3d_\Omega ^{1+n/2} \|\nabla v\|_{2,\Omega} +\left| \int_\Omega v\dif{x}\right|\right),\quad \forall v\in H^{1}(\Omega),
\]
 where \(\Omega\) stands for a bounded convex domain with diameter \(d_\Omega\), 
and for Poincar\'e-type inequalities in \(W^{1,p}_\omega(\Omega)\), where \(1<p<\infty\),
\(\omega \subset\Omega\subset \mathbb{R}^n\) (\(n=2,3\)) has positive measure and \(\Omega\) is a bounded domain
such that is star-shaped with respect to \(\omega\).

To precise the Poincar\'e constants \(C_\Omega\), we state the following proposition  in which the constants are explicitly established in accordance with a 3D domain.
\begin{proposition}\label{proppoincare}
Let \(\Omega_\mathrm{f}\) and  \(\Omega_\mathrm{p}\) be the fluid bidomain and the porous domain, respectively. 
If \(r>1\), then
\begin{align}\label{poincare2}
\| v\|_{2,\Omega_\mathrm{f}} &\leq  \frac{L}{\sqrt{2}}  \|\nabla v \|_{2,\Omega_\mathrm{f}},\quad\forall v\in V(\Omega_f);\\
\|v\|_{2,\Omega_\mathrm{a}\cup\Omega_\mathrm{c}}&\leq  \left(
2\|v\|_{2,\Gamma_\mathrm{w}}^2 + L^2\| \nabla v\|_{2,\Omega_\mathrm{a}\cup\Omega_\mathrm{c}}^2\right)^{1/2}
 ,\quad \forall v\in H^1 (\Omega_\mathrm{a}\cup\Omega_\mathrm{c}); \label{poincareva}\\
\label{poincarer}
\|v\|_{r,\Omega_\mathrm{p}} &\leq  \frac{l_a+l_m+l_c}{r^{1/r}}  \|\nabla v\|_{r,\Omega_\mathrm{p}},\quad\forall v\in V_r(\Omega_p).
\end{align}
Moreover, the following inequalities
\begin{align}\label{Gammai}
\int_{\Gamma_i} |v|^2\dif{s} \leq l_i \int_{\Omega_i}|\nabla v|^2 \dif{x} ; \\
\label{Gammai1}
\int_{\Gamma_i} |v|\dif{s} \leq  |\Omega_i|^{1/2} \|\nabla v\|_{2,\Omega_i}
\end{align}
hold for \(i=\mathrm{a}, \mathrm{c}\).
\end{proposition}
\begin{proof}
The proof is standard by applying the fundamental theorem of calculus, by making recourse of the density of \(C^1\)-functions in \(H^1\). 

For every \( v\in V(\Omega_f)\), we have \(v= 0 \)  a.e.  on \(  \Gamma_\mathrm{in}  \), then we find
\begin{equation}\label{poincarev}
|v(x,y,z)|^2 = \left|\int_{0}^y \partial_y v(x,t,z) \dif{t}\right|^2\leq y \int_{0}^{L}  |\partial_y  v|^2\dif{t},
\quad \forall (x,y,z)\in \Omega_\mathrm{f},
\end{equation}
taking the Cauchy--Schwarz  inequality into account. Hence, integrating over \(\Omega_\mathrm{f}\) we obtain \eqref{poincare2}.

In the domain \(\Omega_i\)(\(i=\) a,c),  we find for \(i=\) a (similarly for \(i=\) c)
\begin{align*}
|v(x,y,z)|^2 &= \left|v(x,0,z) + \int_{0}^y \partial_y v(x,t,z) \dif{t}\right|^2\\
&\leq 2\left(v^2(x,0,z) +y \int_{0}^{L}  |\partial_y  v|^2\dif{t}\right),
\quad \forall (x,y,z)\in \Omega_\mathrm{a},
\end{align*}
taking  the  inequality  \((a+b)^2\leq 2(a^2+b^2)\),  for all \(a,b\geq 0\), 
and the Cauchy--Schwarz  inequality into account. Hence, integrating over \(\Omega_\mathrm{a}\) we obtain \eqref{poincareva}.

For every \(v\in V_r(\Omega_p)\), \(v(x_a-l_a,y,z)=0\) for a.e. \( (y,z)\in ]0,L[\times ]0,H[\), then we have
\begin{align*}
|v(x,y,z)|^r& = \left|\int_{x_a-l_a}^x \partial_x v(t,y,z) \dif{t}\right|^r\\
&\leq (x- (x_a-l_a))^{r-1} \int_{x_a-l_a}^{x_c+l_c}  |\partial_x v|^r\dif{t},
\quad \forall (x,y,z)\in \Omega_\mathrm{p},
\end{align*}
taking the H\"older inequality into account. Hence, integrating over \(\Omega_\mathrm{p}\) we obtain \eqref{poincarer}.
Observe that \((x_c+l_c) - (x_a-l_a)=l_a+l_m+l_c\) and 
\[
\int_{x_a-l_a}^{x_c+l_c} (x- (x_a-l_a))^{r-1}\dif{x} = \frac{(l_a+l_m+l_c)^r}{r}.
\]

In the domain \(\Omega_i\)(\(i=\) a,c), if \(v=0\) a.e. on \(\Gamma\), \textit{i.e.}  at \(x=l_f=x_a-l_a\) and \(x=x_c+l_c\).
Then, we find for \(i=\) a (similarly for \(i=\) c)
\[
|v(x_a,y,z)|^2 =\left|\int_{x_a-l_a}^{x_a} \partial_x v(t,y,z)\dif{t}\right|^2 \leq l_a \int_{x_a-l_a}^{x_a}  |\partial_x v|^2\dif{t},\]
 taking the Cauchy--Schwarz  inequality into account. Hence, integrating over \(]0,L[\times [0,H[\) we obtain \eqref{Gammai}.

Consequently, it follows \eqref{Gammai1} by observing that \(|\Omega_i|=l_i |\Gamma_i|\),
which concludes the proof of Proposition \ref{proppoincare}.
   \end{proof}

\begin{remark}
The estimate \eqref{poincare2} is also valid for vector-valued functions due to the Euclidean norm.
Recall that the notation \(\dif{x}\) refers to the 2D \(\dif{x}\dif{y}\) and the 3D \(\dif{x}\dif{y}\dif{z}\).
\end{remark}

Next,  we precise the required Poincar\'e--Sobolev constants, according to the domain \(\Omega_\mathrm{f}\), for the two-dimensional space.
\begin{proposition}[n=2]\label{propn2}
For every \(v\in H^1(\Omega_\mathrm{f})\) such that
\begin{description}
\item[(i)]
if  \(v(0,x_2)=v(x_1,0)=0\) for all \((x_1,x_2)\in ]0,l_f[\times]0,L[\), then
\begin{align}\label{casei2}
\|v\|_{4,\Omega_\mathrm{f}}^2\leq \frac{ \sqrt{ l_f L}}{2} \|\nabla v\|_{2,\Omega_\mathrm{f}}^2.
\end{align}
\item[(ii)]
if  \(v(x_2=0)=0\), then
\begin{align}\label{caseii2}
\|v\|_{4,\Omega_\mathrm{f}}^2\leq \|v\|_{2,\Gamma_\mathrm{w}}^2+\max\{l_f,L\} \|\nabla v\|_{2,\Omega_\mathrm{f}}^2.
\end{align}
\end{description}
\end{proposition}
\begin{proof}
Both estimates may be proved in half domain \(]0,l_f[\times ]0,L[\).  Analogous proofs can be done in the remaining domain \(\Omega_\mathrm{f}\).

\textbf{Case (i)}
 We use  the fundamental theorem of calculus 
\begin{align}
v(x_1,x_2) &=  \int_0^{x_1}\partial_1 v(t,x_2)\dif{t}\label{ftc1} \\
& = \int_0^{x_2}\partial_2 v(x_1,t)\dif{t}.\label{ftc2}
\end{align}
Arguing as in  \eqref{poincarev}  we have
\begin{align} 
J=\int_0^L\int_0^{l_f} v^4\dif{x_1}\dif{x_2} \leq \int_0^{l_f} \max_{0\leq x_2\leq L}v^2\dif{x_1} \int_0^L \max_{0\leq x_1\leq l_f} v^2\dif{x_2} \nonumber\\
\leq Ll_f \int_0^L\int_0^{l_f} |\partial_2v|^2 \dif{x_1}\dif{x_2}\int_0^L\int_0^{l_f} |\partial_1 v|^2\dif{x_1}\dif{x_2}.\label{Ll2}
\end{align}
Therefore,
\begin{align*}
J^{1/2} &\leq  \sqrt{l_f L} \left(  \int_0^L\int_0^{l_f} |\partial_2v|^2\dif{x_1}\dif{x_2}\right)^{1/2} \left(  \int_0^L\int_0^{l_f} |\partial_1v|^2\dif{x_1}\dif{x_2}\right)^{1/2}\\
&\leq \frac{\sqrt{l_fL}}{2} \int_0^L\int_0^{l_f}(|\partial_1 v|^2+ |\partial_2v|^2) \dif{x_1}\dif{x_2},
\end{align*}
which concludes the proof of case (i).

\textbf{Case (ii)}
 We use  the fundamental theorem of calculus as follows
\begin{align}
v^2(x_1,x_2) &= \left(v(0,x_2)+ \int_0^{x_1}\partial_1 v(t,x_2)\dif{t} \right)^2 \label{2d1} \\
& =\left( \int_0^{x_2}\partial_2 v(x_1,t)\dif{t}\right)^2 \leq L  \int_0^{L}|\partial_2 v(x_1,t)|^2 \dif{t} .\label{2d2}
\end{align}
Adapting the argument  in \eqref{Ll2} but here with \eqref{2d1}-\eqref{2d2} we have
\begin{align*}
J=\int_0^L\int_0^{l_f} v^4\dif{x_1}\dif{x_2} \leq \int_0^{l_f} \max_{0\leq x_2\leq L}v^2\dif{x_1} \int_0^L \max_{0\leq x_1\leq l_f} v^2\dif{x_2} \nonumber\\
\leq 2L \int_0^L\int_0^{l_f} |\partial_2v|^2 \dif{x_1}\dif{x_2}\left( \int_0^L v^2(0,x_2)\dif{x_2}+l_f\int_0^L\int_0^{l_f} |\partial_1 v|^2 \dif{x_1}\dif{x_2}\right).
\end{align*}
The first term is estimated by \eqref{2d2}, while to estimate the second term we apply  the  inequality  \((a+b)^2\leq 2(a^2+b^2)\),  for all \(a,b\geq 0\), and then
the \(l_f\)-version of \eqref{2d2}.

 Next, using the  inequality  \(2ab\leq a^2+b^2\)  for all \(a,b\geq 0\), we obtain
\begin{align*}
J^{1/2}
\leq 2\left(L\int_0^L\int_0^{l_f} |\partial_2v|^2 \dif{x}\right)^{1/2}
\left( \int_0^L v^2(0,x_2)\dif{x_2}+l_f \int_0^L\int_0^{l_f} |\partial_1 v|^2 \dif{x_1}\dif{x_2}\right)^{1/2} \\
\leq   L \int_0^L\int_0^{l_f} |\partial_2v|^2 \dif{x_1}\dif{x_2} +
 \int_{\Gamma_\mathrm{w}} v^2\dif{s}+ l_f\int_0^L\int_0^{l_f} |\partial_1v|^2 \dif{x_1}\dif{x_2} .\qquad
\end{align*}
This last inequality yields \eqref{caseii2}, which concludes the proof of Proposition \ref{proppoincare}.
  \end{proof}

\begin{remark}
The argument of Ladyzhenskaya \cite[pp.8-11]{lady} works for Sobolev inequalities in the form
\begin{align}\label{lady2}
\|v\|_{4,\mathbb{R}^2}^4\leq \varepsilon \| \nabla v\|_{2,\mathbb{R}^2}^4+\frac{1}{\epsilon}\| v\|_{2,\mathbb{R}^2}^4 ;\\
\|v\|_{4,\mathbb{R}^2}^4\leq 3\varepsilon \| \nabla v\|_{2,\mathbb{R}^3}^4+\frac{1}{\epsilon}\| v\|_{2,\mathbb{R}^3}^4,
\end{align}
 for any \(\varepsilon >0\), for smooth functions that decay at infinity. 
Adapting the argument of  \cite[Lemma 1]{lady} for our domain, using  the fundamental theorem of calculus 
\begin{align*}
v^2(x_1,x_2) &=  2\int_0^{x_1}v(t,x_2)\partial_1 v(t,x_2)\dif{t} \\
& = 2\int_0^{x_2}v(x_1,t)\partial_2 v(x_1,t)\dif{t}
\end{align*}
instead \eqref{ftc1}-\eqref{ftc2}
we obtain
\[
\|v\|_{4,\Omega_\mathrm{f}}^2\leq  \sqrt{ l_f L}\|\nabla v\|_{2,\Omega_\mathrm{f}}^2.
\]
Clearly, this constant is worse than the one  obtained in \eqref{casei2}.
For reader's convenience, \(J\) in \eqref{Ll2} reads
\[
J
\leq 4 \int_0^L\int_0^{l_f} |v\partial_2v|\dif{x_1}\dif{x_2}\int_0^L\int_0^{l_f} |v\partial_1 v|\dif{x_1}\dif{x_2}.
\]
Hence,  each term is analyzed making recourse to the Poincar\'e inequality \eqref{poincare2},  in which the domain is considered.
To estimate the first term,  we take the Cauchy--Schwarz inequality into account
\begin{align*}
\int_0^L\int_0^{l_f} |v\partial_2v|\dif{x_1}\dif{x_2} &\leq \left( \int_0^L\int_0^{l_f} |v|^2\dif{x_1}\dif{x_2}\right)^{1/2}
\left( \int_0^L\int_0^{l_f} |\partial_2v|^2\dif{x_1}\dif{x_2}\right)^{1/2} \\
&\leq \frac{L}{\sqrt{2}}  \int_0^L\int_0^{l_f} |\partial_2v|^2\dif{x_1}\dif{x_2} .
\end{align*}
Analogously to estimate the second term.
\[
\int_0^L\int_0^{l_f} |v\partial_1v|\dif{x_1}\dif{x_2}\leq \frac{l_f}{\sqrt{2}}  \int_0^L\int_0^{l_f} |\partial_1v|^2\dif{x_1}\dif{x_2} .
\]
\end{remark}

For the trilinear convective term, we establish the following quantitative estimates for the two-dimensional space.
\begin{lemma}\label{lemae}
For each \( \mathbf{v}\in  \mathbf{H}^{1}(\Omega_\mathrm{f})\),  the following functional is  well defined and  continuous: 
 \(
e\in H^1(\Omega_\mathrm{f})\mapsto\int_{\Omega_\mathrm{f}} e v\nabla\cdot \mathbf{v} \dif{x},\)
for all \(v \in H^1(\Omega_\mathrm{f}).\) Moreover, 
\begin{enumerate}
\item the quantitative estimate
\begin{equation}\label{advte}
\left|\int_{\Omega_ \mathrm{f}}ev\nabla\cdot \mathbf{v}  \dif{x} \right|\leq  \sqrt{2L}
\left( \|e\|_{2,\Gamma_\mathrm{w}}^2+l_f   \| \nabla e\|_{2,\Omega_\mathrm{f}} ^2\right)^{1/2}
 \| \nabla v\|_{2,\Omega_\mathrm{f}}  \|\nabla\cdot \mathbf{v}\|_{2,\Omega_\mathrm{f}} 
\end{equation}
holds for any  \(v \in H^1_{in}(\Omega_\mathrm{f})\).
\item  the quantitative estimate
\begin{equation}\label{advtev}
\left|\int_{\Omega_ \mathrm{f}}ev\nabla\cdot \mathbf{v}  \dif{x} \right|\leq  \sqrt{L}
\left( \|e\|_{2,\Gamma_\mathrm{w}}^2+L  \| \nabla e\|_{2,\Omega_\mathrm{f}} ^2\right)^{1/2}
 \| \nabla v\|_{2,\Omega_\mathrm{f}}  \|\nabla\cdot \mathbf{v}\|_{2,\Omega_\mathrm{f}} 
\end{equation}
holds for any  \(e,v \in H^1_{in}(\Omega_\mathrm{f})\).
\end{enumerate}
\end{lemma}
\begin{proof}
Both estimates may be proved in half domain \(]0,l_f[\times ]0,L[\).  Analogous proofs can be done in the remaining domain \(\Omega_\mathrm{f}\).
 We use  the fundamental theorem of calculus as follows
\begin{align}
v(x_1,x_2) & =  \int_0^{x_2}\partial_2 v(x_1,t)\dif{t}; \label{2dv}\\
\mbox{Case (1)} \quad
e^2(x_1,x_2)&= \left(e(0,x_2)+ \int_0^{x_1}\partial_1 e(t,x_2)\dif{t} \right)^2 ;  \label{2de}  \\
\mbox{Case (2)}\quad 
e^2(x_1,x_2) & =  \left(\int_0^{x_2} \partial_2 e(x_1,t)\dif{t}\right)^2 ; \label{e2case2}\\
e^2(x_1,x_2)&=e^2(0,x_2)+2 \int_0^{x_1}e \partial_1 e(t,x_2)\dif{t} . \label{e2wcase2}
\end{align}

We proceed as follows. We firstly apply \eqref{2dv} for \(v\) and  the Cauchy--Schwarz inequality for the integral in \(x_2\),  obtaining
\begin{align*}
I&=  \int_0^L\int_0^{l_f} |ev \nabla\cdot\mathbf{v}| \dif{x}  \\
&\leq \int_0^{l_f}  \left[\int_0^L|\partial_2 v|\dif{x_2} \left(\int_0^L | e|^2 \dif{x_2}\right)^{1/2}
\left(\int_0^L| \nabla\cdot\mathbf{v}|^2 \dif{x_2}\right)^{1/2} \right] \dif{x_1}. 
\end{align*}

\textbf{Case (1).}
Secondly we use  \eqref{2de} for \(e\),  the inequality \((a+b)^2\leq 2(a^2+b^2)\),  for all \(a,b\geq 0\), 
and  the Cauchy--Schwarz inequality for the integral in \(x_1\),  obtaining
\begin{align}
I\leq & \left( 2\int_0^L \left( e^2(0,x_2)+\left|  \int_{0}^{l_f}|\partial_1 e |\dif{x_1}\right|^2  \right) \dif{x_2}\right)^{1/2}\nonumber \\
&\times \left(\int_0^{l_f} \left| \int_0^L |\partial_2 v|\dif{x_2} \right|^{2}\dif{x_1}\right)^{1/2} 
  \left(\int_0^{l_f}\int_0^L|  \nabla\cdot\mathbf{v}|^2 \dif{x}\right)^{1/2} .\label{cotaIev1}
\end{align}
Next, using the Cauchy--Schwarz inequality
\begin{align}
\left| \int_0^L |\partial_2 v|\dif{x_2} \right|^2 &\leq  L  \int_0^{L}|\partial_2 v |^2 \dif{x_2} ; \label{lfv}\\
\left|\int_0^{l_f}| \partial_1 e | \dif{x_1}\right|^2 &\leq l_f \int_0^{l_f}| \partial_1 e|^2 \dif{x_1},\label{lfe}
\end{align}
and substituting in \eqref{cotaIev1} we conclude \eqref{advte}. 

\textbf{Case (2).}
Secondly we use \eqref{e2wcase2}  for \(e\)
and  the Cauchy--Schwarz inequality for the integral in \(x_1\),  obtaining
\begin{align}
I\leq & \left( \int_0^L \left( e^2(0,x_2)+ 2\int_{0}^{l_f}|e\partial_1 e |\dif{x_1}  \right) \dif{x_2}\right)^{1/2}\nonumber \\
&\times \left(\int_0^{l_f} \left| \int_0^L |\partial_2 v|\dif{x_2} \right|^{2}\dif{x_1}\right)^{1/2} 
  \left(\int_0^{l_f}\int_0^L|  \nabla\cdot\mathbf{v}|^2 \dif{x}\right)^{1/2} .\label{cotaIev2}
\end{align}
Next, using the Cauchy--Schwarz inequality twice and again after using \eqref{e2case2}, we find
\begin{align*}
\left| \int_0^L |\partial_2 v|\dif{x_2} \right|^2 &\leq  L  \int_0^{L}|\partial_2 v |^2 \dif{x_2} ;\\
\int_0^{l_f}|e \partial_1 e | \dif{x_1}  &\leq\left( \int_0^{l_f}|  e|^2 \dif{x_1}\right)^{1/2}  \left( \int_0^{l_f}| \partial_1 e|^2 \dif{x_1} \right)^{1/2} \\
&\leq \left(  L \int_0^{l_f}\int_0^L | \partial_2 e|^2 \dif{x_2}\dif{x_1} \right)^{1/2} \left(\int_0^{l_f}| \partial_1 e|^2 \dif{x_1} \right)^{1/2} \\
&\leq\frac{ 1}{2}  \left( L  \int_0^{l_f}\int_0^L | \partial_2 e|^2 \dif{x_2}\dif{x_1} +\int_0^{l_f}| \partial_1 e|^2 \dif{x_1} \right).
\end{align*}
Finally, we apply  the inequality \(2ab\leq a^2+b^2\) to obtain the Euclidean norm.
Substituting the above inequalities  in \eqref{cotaIev2},
 we conclude \eqref{advtev}, which finishes the proof of Proposition \ref{lemae}.
  \end{proof}

Finally,  the transport term is precised for some   exponent \(q\).
Remind that   \( \Omega_\mathrm{f} \subset\mathbb{R}^n\) is two disjoint bounded Lipschitz domains. 
\begin{lemma}\label{lemaw}
For each \( \mathbf{w}\in  \mathbf{H}^{1}(\Omega_\mathrm{f})\) being such that
\( \mathbf{w}\cdot\mathbf{n}=0\) on \(\Gamma_\mathrm{w} \) and \( \mathbf{w}=u_\mathrm{in}\mathbf{e}_2\) on \(\Gamma_\mathrm{in} \),
  the following functional is  well defined and  continuous: 
 \(
e\in H^1(\Omega_\mathrm{f})\mapsto\int_{\Omega_\mathrm{f}} \mathbf{w}\cdot\nabla e v \dif{x},\)
for all \(v \in H^1(\Omega_\mathrm{f}).\) Moreover, 
\begin{enumerate}
\item the relation
\begin{equation}\label{advt1}
\left|\int_{\Omega_\mathrm{f}} \mathbf{w}\cdot\nabla e v \dif{x} \right|\leq \|\mathbf{w}\|_{q,\Omega_\mathrm{f}} \|\nabla e\|_{2,\Omega_\mathrm{f}}
 \| v\|_{2^*,\Omega_\mathrm{f}} 
\end{equation}
holds for any \(e ,v\in H^1(\Omega_\mathrm{f})\) and  \(q=n>2\)  or \(q>n=2\). Here,
 \(2^*\) denotes  the critical Sobolev exponent if \(n>2\),  that is,  of the Sobolev embedding \(H^1(\Omega)\hookrightarrow L^{2^*}(\Omega)\).
For the sake of simplicity, we also denote by \(2^*\)  any arbitrary real number greater than one, if \(n=2\).
\item if \(n=2\), the quantitative estimate
\begin{equation}\label{advt2}
\left|\int_{\Omega_\mathrm{f}} \mathbf{w}\cdot\nabla v v \dif{x} \right|\leq  (1/2 + \sqrt{2})\sqrt{L}\left(
\|\mathbf{w}_T\|_{2,\Gamma}^2 + l_f \|\nabla \mathbf{w}\|_{2,\Omega_\mathrm{f}} ^2 \right)^{1/2}
 \| \nabla v\|_{2,\Omega_\mathrm{f}} ^2
\end{equation}
holds for any  \(v \in V(\Omega_\mathrm{f})\).
\end{enumerate}
\end{lemma}
\begin{proof}
The relation \eqref{advt1} is consequence of the H\"older inequality,  for \(1/q+1/2^*=1/2\) \textit{i.e.} \(2q/(q-2)=2^*\),
with  \(q=n>2\)  or \(q>n=2\)
to guarantee \( \mathbf{H}^{1}(\Omega_\mathrm{f})\hookrightarrow  \mathbf{L}^{q}(\Omega_\mathrm{f})\).

To prove \eqref{advt2}, instead the direct application of \eqref{advt1} with abstract constants of Sobolev and Poincar\'e
\[
\left|\int_{\Omega_\mathrm{f}} \mathbf{w}\cdot\nabla e v \dif{x} \right|\leq S^*C_\Omega \|\mathbf{w}\|_{q,\Omega_\mathrm{f}}
 \|\nabla e\|_{2,\Omega_\mathrm{f}} \| \nabla v\|_{2,\Omega_\mathrm{f}} ,
\]
where 
  \(S^*\) denotes  the continuity constant of the Sobolev embedding \(H^1(\Omega)\hookrightarrow L^{2^*}(\Omega)\)
and   \(C_{\Omega}\) denotes the Poincar\'e constant,
we analyze,  term by term, the integral (if \(n=2\))
\[
\int_0^L\int_0^{l_f} \mathbf{w}\cdot\nabla v v \dif{x} = \int_0^L\int_0^{l_f} (w_1\partial_1 v+ w_2\partial_2 v )v \dif{x} ,
\]
analogous for \(  \int_0^L \int_{l_a+l_m+l_c}^{l_a +l_m +l_c+l_f} \),
considering the assumptions
\begin{enumerate}
\item  \(w_1(x_1=0)=w_1(x_2=0)=0\) and \(v(x_2=0)=0\);
\item \(w_2(x_2=0)=u_\mathrm{in}\) and \(v(x_2=0)=0\).
\end{enumerate}
In the sequel, we use the notation  \(\dif{x}\)  to the 2D \(\dif{x_1}\dif{x_2}\).

\textsf{Term 1.} We firstly apply \eqref{ftc2} for \(v\) and  the Cauchy--Schwarz inequality for the integral in \(x_2\), 
secondly \eqref{ftc1} for \(w_1\)  and again the Cauchy--Schwarz inequality but now   for the integral in \(x_1\).
Next, we apply  the Cauchy--Schwarz inequality twice  for the appearance of \(l_f\) and \(L\), and finally the inequality \(2ab\leq a^2+b^2\) to obtain the Euclidean norm.
That is,
\begin{align*}
 \int_0^L\int_0^{l_f} |w_1\partial_1 v v| \dif{x}  \leq \int_0^{l_f}  \left[\int_0^L|\partial_2 v|\dif{x_2} \left(\int_0^L | w_1|^2 \dif{x_2}\right)^{1/2}
\left(\int_0^L| \partial_1 v|^2 \dif{x_2}\right)^{1/2} \right] \dif{x_1}  \\
\leq\left(\int_0^L \left|\int_0^{l_f} |\partial_1 w_1 | \dif{x_1} \right|^2\dif{x_2}\right)^{1/2} 
\left(\int_0^{l_f} \left| \int_0^L |\partial_2 v|\dif{x_2} \right|^{2}\dif{x_1}\right)^{1/2}   \left(\int_0^{l_f}\int_0^L| \partial_1 v|^2 \dif{x}\right)^{1/2} \\
\leq\left(l_f \int_0^L \int_0^{l_f}| \partial_1 w_1 |^2 \dif{x}\right)^{1/2} 
\left(L\int_0^{l_f} \int_0^L |\partial_2 v|^2 \dif{x}  \right)^{1/2}\left(\int_0^{l_f} \int_0^L| \partial_1 v|^2 \dif{x}  \right)^{1/2}\\
\leq \frac{\sqrt{L}}{2} \left(l_f \int_0^L \int_0^{l_f}| \partial_1 w_1 |^2 \dif{x}\right)^{1/2} 
\left(\int_0^{l_f} \int_0^L (| \partial_1 v|^2+  |\partial_2 v|^2) \dif{x_2}\dif{x_1}  \right).\qquad
\end{align*}

\textsf{Term 2.} Analogously, we proceed for the second term  firstly applying \eqref{ftc2} for \(v\) and  the Cauchy--Schwarz inequality   for the integral in \(x_2\), obtaining
\begin{align*}
I&= \int_0^L\int_0^{l_f} |w_2\partial_2 v v | \dif{x} \\
&\leq \int_0^{l_f}  \left[\int_0^L|\partial_2 v|\dif{x_2} \left(\int_0^L | w_2|^2 \dif{x_2}\right)^{1/2}
\left(\int_0^L| \partial_2 v|^2 \dif{x_2}\right)^{1/2} \right] \dif{x_1}. 
\end{align*}
Secondly, for \(w_2\) we use the following version of  \eqref{2d1}:
\begin{align*}
w_2^2(x_1,x_2) &= \left(w_2(l_f,x_2)+ \int_{l_f}^{x_1}\partial_1 w_2 (t,x_2)\dif{t} \right)^2 \\
&\leq 2\left( w_2^2(l_f,x_2)+l_f  \int_{0}^{l_f}|\partial_1 w_2 (t,x_2)|^2\dif{t} \right).
\end{align*}
where we apply the inequality \((a+b)^2\leq 2(a^2+b^2)\) and  the Cauchy--Schwarz inequality   for the appearance of \(l_f\).  
Finally, we substitute the above inequality and simultaneously
we apply  the Schwarz inequality twice  in the integral of \(x_1\) and again in the integral of \(x_2\)  for the appearance of \(L\).
That is, 
\begin{align*}
I\leq\left( 2\int_0^L \left( w_2^2(l_f,x_2)+l_f  \int_{0}^{l_f}|\partial_1 w_2 (t,x_2)|^2\dif{t} \right) \dif{x_2}\right)^{1/2} 
L ^{1/2}\int_0^{l_f}\int_0^L| \partial_2 v|^2 \dif{x} \\
\leq \sqrt{2L}  \left(\int_0^L  w_2^2(l_f,x_2) \dif{x_2} + l_f \int_0^L \int_0^{l_f}| \partial_1 w_2 |^2 \dif{x}\right)^{1/2} 
\left(\int_0^{l_f} \int_0^L| \partial_2 v|^2 \dif{x}  \right).
\end{align*}

Then, we conclude \eqref{advt2}
by summing
\[
 \int_0^L\left(\int_0^{l_f} +\int_{l_a +l_m +l_c}^{l_a +l_m +l_c+l_f}\right)  | \mathbf{w}\cdot\nabla v v|\dif{x}
\leq C_{PS} \left(
\|\mathbf{w}_T\|_{2,\Gamma}^2 + l_f \|\nabla \mathbf{w}\|_{2,\Omega_\mathrm{f}} ^2 \right)^{1/2} \| \nabla v\|_{2,\Omega_\mathrm{f}} ^2,
\]
with \(C_{PS} =(1/2 + \sqrt{2}) \sqrt{L} \).
  \end{proof}

\section{Existence of auxiliary solutions}
\label{auxtdt}

The existence of a unique weak solution \((\mathbf{U},p)=( \mathbf{U},p) (\pi, \bm{\varrho},\xi)\)
 to the variational equality \eqref{wvfup}  can be stated under the assumption of 
 \(\bm{\varrho}\in [ L^4(\Omega_\mathrm{f}) ]^{2}\)  and \(n=2,3\) \cite{pemfc}.
Faced with Lemma \ref{lemae} we establish its existence  as follows.
\begin{proposition}[Auxiliary velocity-pressure pair]\label{proppxi}
Let \(\pi \in L^2(\Omega_\mathrm{p})\),
 \(\bm{\varrho}\in [ H_{in}^1(\Omega_\mathrm{f}) ]^{2}\) and  \(\xi\in H^1(\Omega)\) be given.
 Under the assumptions (H1), (H4) and (H7),
the Dirichlet--BJS/Stokes--Darcy problem \eqref{wvfup} admits  a unique weak solution
 \((\mathbf{U},p)\in \mathbf{V}(\Omega_f)\times ( H(\Omega_p)/\mathbb{R}) \).
Moreover, if  \(n=2\),  the quantitative estimate for  \(\mathbf{u}=\mathbf{U} +\mathbf{u}_0 \)
\begin{align}\label{cotaup}
\frac{\mu_\#}{2C_K} \|  \nabla\mathbf{u} \|_{2,\Omega_\mathrm{f} }^2+ \beta_\#\|\mathbf{u}_T\|_{2,\Gamma} ^2 +
\frac{K_l}{\mu^\#} \|  \nabla p\|_{2,\Omega_\mathrm{p}}^2 \nonumber\\ 
\leq \left( \sqrt{2L} \frac{R_M }{\sqrt{\mu_\#}}
\left( \|\xi\|_{2,\Gamma_\mathrm{w}}^2+l_f   \| \nabla \xi\|_{2,\Omega_\mathrm{f}} ^2\right)^{1/2}
 \| \nabla \bm{\varrho}\|_{2,\Omega_\mathrm{f}} +C_0 \right)^2
\end{align}
holds, with \(C_K>1\) being the Korn constant and \(C_0\) being defined by
\begin{equation}
C_0:= \sqrt{\mu^\#}\|D\mathbf{u}_0\|_{2,\Omega_\mathrm{f}} 
+ \frac{\lambda^\#}{\sqrt{\mu_\#}}  \|\nabla\cdot\mathbf{u}_0\|_{2,\Omega_\mathrm{f}}  .
\label{C0}
\end{equation}
\end{proposition}
\begin{proof}
 The existence of a unique weak solution \((\mathbf{U},p)\in \mathbf{V}(\Omega_f)\times (H(\Omega_p)/\mathbb{R})\)
  to the variational equality \eqref{wvfup} is obtained by the Lax--Milgram lemma (for details see \cite{pemfc}).

The quantitative estimate \eqref{cotaup} follows from taking \((\mathbf{v},v)=(\mathbf{U},p)\) as a test function in \eqref{wvfup}. Indeed, 
we take the H\"older and Young inequalities into account,  apply  the assumptions  \eqref{mu}-\eqref{defkl},  \eqref{gamm},  and (H7),
and use the Korn inequality \eqref{korn}. Faced with \(n=2\),  Lemma \ref{lemae} (1) concludes the  quantitative estimate \eqref{cotaup}.
  \end{proof}

The continuous dependence can be established as follows, which proof may be found in  \cite{pemfc}.
\begin{proposition}[Continuous dependence]\label{upm}
Suppose that the assumptions of Proposition \ref{proppxi} are fulfilled.
Let \(\{\pi_m\}\),  \(\{\bm{\varrho}_m\}\)   and \(\{\xi_m\}\) be sequences such that  
 \(\pi_m\rightarrow\pi\) in \(L^2(\Omega_\mathrm{p})\),
\(\bm{\varrho}_m\rightarrow\bm{\varrho} \) in \([L^4(\Omega_\mathrm{f})]^2\),
and \(\xi_m\rightharpoonup \xi\) in \(H^1(\Omega)\), respectively.
If \((\mathbf{u}_m,p_m)= (\mathbf{U}+\mathbf{u}_0,p)(\pi_m, \bm{\varrho}_m, \xi_m)\) are the unique solutions to \eqref{wvfup}\(_m\),
 then \begin{align}\label{um}
\mathbf{U}_m \rightharpoonup \mathbf{U} \mbox{ in }\mathbf{V}(\Omega_f);\\
p_m\rightharpoonup p\mbox{ in }H(\Omega_p),\label{pm}
\end{align}
with \((\mathbf{u},p)= (\mathbf{U}+\mathbf{u}_0,p)(\pi, \bm{\varrho}, \xi)\) being the solution to \eqref{wvfup}.
\end{proposition}

The existence of a unique weak  solution 
\(( \bm{\Upsilon},\Theta, \phi_{cc})=(\bm{\rho},  \theta, \phi) (\mathbf{w},  \bm{\varrho},\xi, \Phi)\)
 to the variational equalities \eqref{aux-pdi1}-\eqref{aux-phicc}  can be stated under the assumption of 
\(\mathbf{w}\in \mathbf{L}^{q}(\Omega_\mathrm{f})\) for 
\(q\geq n>2 \) or \(q>n=2\), 
\(\Phi\in L^t(\Omega_\mathrm{a}\cup\Omega_\mathrm{c})\),  with
\(t\geq 2n/(n+2) \) if \(n>2\) or  \( t>1 \) if \( n=2\),
 and \(n=2,3\) \cite{pemfc}.
Faced with Lemma \ref{lemaw} we establish its 2D existence as follows.
\begin{proposition}[Auxiliary partial density-temperature-potential triplet]\label{proptt}
Let  \( n=2\) and \(\theta_e\in L^2(\Gamma_\mathrm{w})\).
Let  \(\mathbf{w}\in \mathbf{H}^{1}(\Omega_\mathrm{f})\) 
be such that 
\begin{equation}\label{defw}
\left(
\|\mathbf{w}_T\|_{2,\Gamma}^2 + l_f \|\nabla \mathbf{w}\|_{2,\Omega_\mathrm{f}} ^2 \right)^{1/2} < \|  \mathbf{w}\|_{1,2,\Omega_\mathrm{f}}  < root_1 
, \end{equation}
 for \(l_f<1\) and \(root_1\) being the positive root of the quadratic polynomial \( 4\min_{i} a_{i,\#}- 2(1 + 2 \sqrt{2})\sqrt{L}t-Lt^2=0\).
Let   \((\varrho_1, \varrho_2,\xi) \in  [H^1(\Omega)]^3\) 
and 
\(\Phi\in L^t(\Omega_\mathrm{a}\cup\Omega_\mathrm{c})\),  with \( t>1 \),
 be given.
Under the assumptions  (H2)-(H3), (H5)-(H6)  and (H8)-(H9),
the  variational problem \eqref{aux-pdi1}-\eqref{aux-phicc} admits a unique solution \((\bm{\Upsilon},\Theta, \phi_{cc})\in  [V(\Omega)]^3\times V(\Omega_p) \).
Moreover, the quantitative estimate
\begin{align}\label{cotatriple}
\sum_{i=1}^2 \left(a_{i,\#} -(1/2 + \sqrt{2})\sqrt{L} \|\mathbf{w}\|_{1,2,\Omega_\mathrm{f}}-\frac{L}{4}\|\nabla \mathbf{w}\|_{2,\Omega_\mathrm{f}} ^2 \right) 
\|\nabla\Upsilon_i\|_{2,\Omega_\mathrm{f}}^2 \nonumber \\
+ \sum_{i=1}^2\min\{  a_{i,\#} ,a_{i,m}\} \|\nabla\Upsilon_i\|_{2,\Omega_\mathrm{p}}^2  
+\min\{ a_{3,\#},a_{3,m}\} \|\nabla\Theta\|_{2,\Omega}^2+  \frac{h_\#}{2} \|\Theta\|_{2,\Gamma_\mathrm{w}}^2 \nonumber \\
 + a_{4,m}  \|  \nabla\phi\|_{2,\Omega_\mathrm{m}} ^2
+\frac{\sigma_{\#}}{2} \|  \nabla\phi\|_{2,\Omega_\mathrm{a}\cup\Omega_\mathrm{c}} ^2 
\leq \frac{(S^*\sigma^\#)^2}{2 k_\#} \|\Phi\|_{t,\Omega_\mathrm{a}\cup\Omega_\mathrm{c}}^2
+  \frac{h^\#}{2}\|\theta_e\|_{2,\Gamma_\mathrm{w}}^2 +\mathcal{B}_0
\end{align}
holds, for \(\phi=\phi_{cc}+E_\mathrm{cell}\chi_{\Omega_\mathrm{c}}\).
 Here, \(S^*=S(\Omega_\mathrm{a}\cup\Omega_\mathrm{c},t')\) and \(\mathcal{B}_0\) is  defined by
\begin{align}
\mathcal{B}_0&:= \sum_{i=1}^2 \left(\frac{D_i^\#}{2}+  \frac{2}{\epsilon_6} \frac{((D'_{i})^\#)^2}{k_\#}\right)\| \nabla\rho_{i,0}\|_{2,\Omega}^2+
 \frac{1}{\epsilon_1}  \sum_{\genfrac{}{}{0pt}{2}{i,j=1}{i\not= j}}^2\frac{(D^\#_{ij})^2}{ D_{i,m}}  \|\nabla \rho_{j,0}\|_{2,\Omega_\mathrm{m}}^2\nonumber  \\
&+ \frac{F^2}{ \epsilon_8 M_1^2}  \frac{ (D_{1,m}^\#)^2}{\sigma_{m,\#}}  \|\nabla \rho_{1,0}\|_{2,\Omega_\mathrm{m}}^2
 + L   \sum_{i=1}^2  \| \rho_{i,0}\|_{\infty,\Omega_\mathrm{f}}^2  \nonumber\\
&+  \frac{ \kappa^\# }{2}  \|\nabla \theta_{0}\|_{2,\Omega}^2+
 \frac{1}{\epsilon_2}  \sum_{i=1}^2\frac{(S^\#_{i})^2}{k_\#}\| \nabla\theta_{0}\|_{2,\Omega}^2 
+ \frac{1}{\epsilon_9} \sigma_m^\#(\alpha^\#)^2\| \nabla\theta_{0}\|_{2,\Omega_\mathrm{m}}^2 .\label{defb0}
\end{align}
For the sake of simplicity, it is assumed that \(D_{i,m}\#\leq D_i^\#\) (\(i=1,2\)),
 \(\epsilon_1=\epsilon_4\) and  \(\epsilon_2=\epsilon_5\). 
\end{proposition}
\begin{proof}
Let \(\mathbf{w}\in \mathbf{H}^1(\Omega_\mathrm{f})\),   \( (\bm{\varrho}, \xi) \in  [ H^1(\Omega) ]^3\)  and 
\(\Phi\in L^t(\Omega_\mathrm{a}\cup\Omega_\mathrm{c})\),  be fixed,  \(t>1\).

 The existence of a unique weak solution  \((\bm{\Upsilon},\Theta, \phi_{cc} )\in  [V(\Omega)]^3 \times V(\Omega_p)\)
 to the variational equalities \eqref{aux-pdi1}-\eqref{aux-phicc}  can be obtained by   the  Browder--Minty Theorem (cf. \cite{pemfc}).
Indeed,  the operator
\(T:  [V(\Omega)]^{3}\times V(\Omega_p)  \rightarrow  \left( [V(\Omega)]^{3}\times V(\Omega_p) \right)'\), defined  by
\begin{align*}
\langle T(\mathbf{Y}),\mathbf{v}\rangle =\int_{\Omega}\mathsf{A}( \bm{\varrho},\xi) \nabla\mathbf{Y}\cdot\nabla \mathbf{v}\dif{x}
+\sum_{i=1}^{2}\int_{\Omega_\mathrm{f}}  Y_i \mathbf{w}\cdot\nabla  v\dif{x} \\
+ \int_{\Gamma_\mathrm{w}} h_c(\xi) Y_3 v \dif{s} 
+ \sum_{\ell=a,c} \int_{\Gamma_\ell} j_{\ell} (Y_{4,\ell}-Y_{4,m}) (w_\ell-w_m)\dif{s},
\end{align*}
  where  \(\mathbf{Y}= (\bm{\Upsilon} ,\Theta, \phi_{cc} )\) and \(\mathbf{v}=(v,v, v,w)\),  
is hemicontinuous, strictly monotone and coercive (see \eqref{cotatriple}), if provided by \eqref{defw}.

Let us establish the quantitative estimate \eqref{cotatriple}.
We take \(v=\Upsilon_1\),  \(v=\Upsilon_2\), 
\(v=\Theta\) and   \(w=\phi_{cc}\) as  test functions in \eqref{aux-pdi1}, \eqref{aux-pdi2},  \eqref{aux-tt} and \eqref{aux-phicc}, respectively. 
Applying the H\"older and  Young inequalities,  and summing the  obtained expressions, we get
\begin{align} \frac{1}{2 } 
\sum_{i=1}^2 \| \sqrt{D_i(\xi)}\nabla\Upsilon_i\|_{2,\Omega}^2
+  \frac{1}{2 }  \|\sqrt{k(\xi)}\nabla\Theta\|_{2,\Omega}^2
+ \frac{1}{2 }  \| \sqrt{h_c(\xi)}\Theta\|_{2,\Gamma_\mathrm{w}}^2  \nonumber\\
+ \| \sqrt{\sigma_m(\varrho_2,  \xi)} \nabla\phi\|_{2,\Omega_\mathrm{m}}^2 
+\| \sqrt{ \sigma (\bm{\varrho},\xi)}\nabla\phi\|_{2,\Omega_\mathrm{a}\cup\Omega_\mathrm{c}}^2
\leq \sum_{i=1}^{4} \mathcal{I}_i + \mathcal{I}_0 \nonumber\\
+  (1/2 + \sqrt{2})\sqrt{L}\left(
\|\mathbf{w}_T\|_{2,\Gamma}^2 + l_f \|\nabla \mathbf{w}\|_{2,\Omega_\mathrm{f}} ^2 \right)^{1/2}
\sum_{i=1}^2\|\nabla \Upsilon_i\|_{2,\Omega_\mathrm{f}}^2 \nonumber\\
+\frac{L}{\sqrt{2}}\| \nabla\mathbf{w}\|_{2,\Omega_\mathrm{f}} \sum_{i=1}^2 \|\rho_{i,0}\|_{\infty,\Omega_\mathrm{f}}  \|\nabla\Upsilon_i\|_{2,\Omega_\mathrm{f}}\nonumber \\
+ \frac{1}{2 } \| \sqrt{h_c(\xi)}(\theta_e-\theta_0)\|_{2,\Gamma_\mathrm{w}}^2
+  \| \sigma (\bm{\varrho},\xi) \Phi\|_{t,\Omega_\mathrm{a}\cup\Omega_\mathrm{c}}\|\Theta\|_{t',\Omega_\mathrm{a}\cup\Omega_\mathrm{c}} 
 \label{cotatriple1}
\end{align} 
taking  \eqref{advt2} to the trilinear term but applying the Poincar\'e inequality \eqref{poincare2} for the corresponding non-homogeneous term.

Here, we consider
\begin{align*}
\mathcal{I}_1 :=&  \frac{1}{\epsilon_1}
\left\| \frac{D_{12}(\xi)}{\sqrt{D_1(\xi)}} \nabla \Upsilon_2 \right\|_{2,\Omega_\mathrm{m}}^2+
 \frac{1}{\epsilon_2} \left\| \frac{\varrho_1 S_{1}(\varrho_1,\xi)}{\sqrt{D_1(\xi)}} \nabla \Theta \right\|_{2,\Omega}^2 
 + \frac{1}{2\epsilon_3}  \frac{1}{R^2} \left\| \frac{ \psi(\varrho_1) \kappa(\xi) }{\xi\sqrt{D_1(\xi)}} \nabla\phi \right\|_{2,\Omega_\mathrm{m}}^2 \\
& +  \frac{\epsilon_1+ \epsilon_3}{2 } \| \sqrt{D_1(\xi)}\nabla \Upsilon_1\|_{2,\Omega_\mathrm{m}}^2 
+  \frac{\epsilon_2}{2 } \| \sqrt{D_1(\xi)}\nabla \Upsilon_1\|_{2,\Omega}^2  ; \\
\mathcal{I}_2 :=&  \frac{1}{\epsilon_4}
\left\| \frac{D_{21}(\xi)}{\sqrt{D_2(\xi)}} \nabla \Upsilon_1 \right\|_{2,\Omega_\mathrm{m}}^2+
 \frac{1}{\epsilon_5} \left\| \frac{\varrho_2 S_{2}(\varrho_2,\xi)}{\sqrt{D_2(\xi)}} \nabla \Theta \right\|_{2,\Omega}^2 \\
& +  \frac{\epsilon_4}{2 } \| \sqrt{D_2(\xi)}\nabla \Upsilon_2 \|_{2,\Omega_\mathrm{m}}^2
 +  \frac{\epsilon_5}{2 } \| \sqrt{D_2(\xi)}\nabla \Upsilon_2 \|_{2,\Omega}^2  ;\\
\mathcal{I}_{3} :=&  \frac{2}{\epsilon_6}\sum_{j=1}^{2} \left(\frac{R}{M_j}\right)^2
\left\| \frac{\xi^2 D'_{j}( \varrho_j,\xi)}{\sqrt{k(\xi)}} \nabla \Upsilon_j \right\|_{2,\Omega}^2 
 + \frac{1}{2\epsilon_7}  \left\| \frac{\Pi(\xi) \sigma_m( \varrho_2 ,\xi)}{\sqrt{k(\xi)}} \nabla\phi \right\|_{2,\Omega_\mathrm{m}}^2\\
&+ \frac{\epsilon_6}{2 } \| \sqrt{k(\xi)}\nabla \Theta \|_{2,\Omega}^2 
+ \frac{\epsilon_7}{2 } \| \sqrt{k(\xi)}\nabla \Theta \|_{2,\Omega_\mathrm{m}}^2;\\
 \mathcal{I}_{4}: =&  \frac{1}{ \epsilon_8 M_1^2}  \left \| \frac{\kappa (\xi)}{ \sqrt{ \sigma_m(\varrho_2,  \xi)} }\nabla \Upsilon_1\right\|_{2,\Omega_\mathrm{m}}^2
+ \frac{1}{\epsilon_9} \| \sqrt{ \sigma_m(\varrho_2,  \xi)} \alpha_\mathrm{S}(\xi) \nabla \Theta \|_{2,\Omega_\mathrm{m}}^2\\
&+ \frac{\epsilon_8+\epsilon_9}{2 } \| \sqrt{\sigma_m(\varrho_2,\xi)}\nabla \phi \|_{2,\Omega_\mathrm{m}}^2,
\end{align*}
for any \(\epsilon_1,\cdots,\epsilon_9>0\) being such that \(\epsilon_1+\epsilon_2+\epsilon_3<1\), \(\epsilon_4+\epsilon_5<1\),
\(\epsilon_6 +\epsilon_7<1\) and \(\epsilon_8 + \epsilon_9 <2\).
In particular,  the proton ionic conductivity \(\kappa= FD_1\)verifies \(|\kappa|\leq F D_{1,m}^\#\).

The last term  \(\mathcal{I}_0\) stand for the nonhomogeneous extensions given in (H8)
\begin{align*}
\mathcal{I}_0 :=& \frac{1}{2 } 
\sum_{i=1}^2 \| \sqrt{D_i(\xi)}\nabla\rho_{i,0}\|_{2,\Omega}^2+
 \frac{1}{\epsilon_1} \left\| \frac{D_{12}(\xi)}{\sqrt{D_1(\xi)}} \nabla \rho_{2,0}\right\|_{2,\Omega_\mathrm{m}}^2
+ \frac{1}{\epsilon_2}\left\| \frac{\varrho_1 S_{1}(\varrho_1,\xi)}{\sqrt{D_1(\xi)}} \nabla \theta_0 \right\|_{2,\Omega}^2    \\
 & +\frac{1}{\epsilon_4} \left\| \frac{D_{21}(\xi)}{\sqrt{D_2(\xi)}} \nabla\rho_{1,0}\right\|_{2,\Omega_\mathrm{m}}^2 
+ \frac{1}{\epsilon_5} \left\| \frac{\varrho_2 S_{2}(\varrho_2,\xi)}{\sqrt{D_2(\xi)}} \nabla \theta_0 \right\|_{2,\Omega}^2   \\
& +  \frac{1}{2 }  \|\sqrt{k(\xi)}\nabla\theta_0\|_{2,\Omega}^2
+\frac{2}{\epsilon_6}\sum_{j=1}^{2} \left(\frac{R}{M_j}\right)^2
\left\| \frac{\xi^2 D'_{j}( \varrho_j,\xi)}{\sqrt{k(\xi)}} \nabla \rho_{j,0} \right\|_{2,\Omega}^2 \\
&+ \frac{1}{ \epsilon_8 M_1^2}  \left \| \frac{\kappa (\xi)}{ \sqrt{ \sigma_m(\varrho_2,  \xi)} }\nabla \rho_{1,0}\right\|_{2,\Omega_\mathrm{m}}^2
+ \frac{1}{\epsilon_9} \| \sqrt{ \sigma_m(\varrho_2,  \xi)} \alpha_\mathrm{S}(\xi) \nabla \theta_0 \|_{2,\Omega_\mathrm{m}}^2.
\end{align*}

We observe that the Sobolev embedding \(V(\Omega)\hookrightarrow L^{t'}(\Omega)\) holds for \(t'>1\), with the corresponding
 Sobolev constant \(S(\Omega,t')\), for  the last term in \eqref{cotatriple1}.

For instance, we may choose  \(\epsilon_1=\epsilon_4\) and  \(\epsilon_2=\epsilon_5\). 
Therefore,  applying  the assumptions \eqref{Dif}-\eqref{Dij} on the left hand side in \eqref{cotatriple1} and
also on  \(\mathcal{I}_1,\cdots,  \mathcal{I}_4\) and \(\mathcal{I}_0\),
we may recourse to the auxiliary parameters \eqref{defai}-\eqref{defa4} to obtain  the estimate \eqref{cotatriple}.

Finally, the assumption \eqref{defw} assures the positiveness 
\(a_{i,\#} -(1/2 + \sqrt{2})\sqrt{L} \|\mathbf{w}\|_{1,2,\Omega_\mathrm{f}}-\frac{L}{4}\|\nabla \mathbf{w}\|_{2,\Omega_\mathrm{f}} ^2>0\).
  \end{proof}

\begin{remark}
The auxiliary parameters \eqref{defai}-\eqref{defa4} are dependent on the construction of the 
quantitative estimate \eqref{cotatriple}, in particular, on the choice of  \(\mathcal{I}_1,\cdots,  \mathcal{I}_4\).
\end{remark}

\begin{corollary}
Under the assumptions of Proposition \ref{proptt},  if \(a_{i,m}\leq a_{i,\#}\) (\(i=1,2\)), 
\(t=2\) and \(L^2<2\), then we have
\begin{align}\label{cotattphi}
\sum_{i=1}^2 \left(a_{i,\#} -(1/2+\sqrt{2})\sqrt{L} \|\mathbf{w}\|_{1,2,\Omega_\mathrm{f}}
-\frac{L}{4}\|\nabla \mathbf{w}\|_{2,\Omega_\mathrm{f}} ^2 \right) \|\nabla\Upsilon_i\|_{2,\Omega_\mathrm{f}}^2 \nonumber\\
+\sum_{i=1}^2 a_{i,m} \|\nabla\Upsilon_i\|_{2,\Omega_\mathrm{p}}^2 
+ a_{3}\|\nabla\Theta\|_{2,\Omega}^2  
+  \frac{h_\#}{2} \|\Theta\|_{2,\Gamma_\mathrm{w}}^2   \nonumber\\
+ a_{4,m} \|  \nabla\phi\|_{2,\Omega_\mathrm{m}} ^2 
+ \frac{\sigma_{\#}}{2} \|  \nabla\phi\|_{2,\Omega_\mathrm{a}\cup\Omega_\mathrm{c}} ^2
\leq \frac{(\sigma^\#)^2}{ k_\#} \|\Phi\|_{2,\Omega_\mathrm{a}\cup\Omega_\mathrm{c}}^2
+  \frac{h^\#}{2}\|\theta_e\|_{2,\Gamma_\mathrm{w} }^2 +\mathcal{B}_0,
\end{align}
where  \(a_3:=\min\{a_{3,\#},a_{3,m}\}\).
\end{corollary}
\begin{proof}
Considering \eqref{poincareva}, the constant \(S^*\) is greatly simplified by \(\sqrt{\max\{2,L^2\}}=\sqrt{2}\) for \(t=2\) and \(L^2<2\).
  \end{proof}

The continuous dependence is established as follows.
\begin{proposition}[Continuous dependence]\label{propttm}
Suppose that the assumptions of Proposition \ref{proptt} are fulfilled.
Let \(\{\mathbf{w}_m\}\),  \(\{\bm{\varrho}_m\}\),  \(\{\xi_m\}\)  and \(\{\Phi_m\}\)  be sequences such that 
\(\mathbf{w}_m\rightarrow \mathbf{w}\) in \(\mathbf{L}^q(\Omega_\mathrm{f})\),  \(q=n>2\)  or \(q>n=2\),  
\({(\varrho_1)}_m\rightharpoonup \varrho_1\) in \(H^1(\Omega)\),  \({(\varrho_2)}_m\rightharpoonup \varrho_2\) in \(H^1(\Omega)\),
 \(\xi_m\rightharpoonup\xi\) in \(H^1(\Omega)\),   
and  \(\Phi_m\rightharpoonup\Phi\) in \( L^t(\Omega_\mathrm{a}\cup\Omega_\mathrm{c})\), respectively.
If \((\bm{\Upsilon} _m,\Theta_m,(\phi_{cc})_m)= (\bm{\rho} ,\theta,\phi)(\mathbf{w}_m, \bm{\varrho}_m, \xi_m,\Phi_m)\) are the unique solutions to
 \eqref{aux-pdi1}\(_m\)-\eqref{aux-phicc}\(_m\),
 then 
\begin{align}
\label{pdm}
\bm{\Upsilon} _m \rightharpoonup \bm{\Upsilon} \mbox{ in } [H^1(\Omega)]^2;\\
\Theta_m\rightharpoonup \Theta\mbox{ in }H^1(\Omega);\label{ttm}\\
\phi_m\rightharpoonup \phi\mbox{ in }H^1(\Omega_\mathrm{p}),\label{ppm}
\end{align}
with \((\bm{\Upsilon} ,\Theta, \phi_{cc})= (\bm{\rho} ,\theta, \phi)(\mathbf{w}, \bm{\varrho}, \xi, \Phi)\) being the solution to \eqref{aux-pdi1}-\eqref{aux-phicc}.
\end{proposition}
\begin{proof}
Let \(\{\mathbf{w}_m\}\),   \(\{\bm{\varrho}_m\}\),  \(\{\xi_m\}\)  and \(\{\Phi_m\}\)  be sequences in the conditions of the proposition,
and let \((\bm{\Upsilon} _m,\Theta_m, (\phi_{cc})_m)\) solve the corresponding  variational system \eqref{aux-pdi1}\(_m\)-\eqref{aux-phicc}\(_m\).
Thanks to the estimate \eqref{cotatriple}, we can extract a (not relabeled) subsequence \(\{(\bm{\Upsilon} _m,\Theta_m, (\phi_{cc})_m)\}\) such that 
the convergences \eqref{pdm}-\eqref{ppm} hold.   
Notice that  the Rellich--Kondrachov embedding \( H^1(\Omega)\hookrightarrow\hookrightarrow L^{p}(\Omega) \) is valid
with exponents \(q\), \(p\) and \(2\) such that 
\[
 \frac{1}{2^*}<\frac{1}{p}=\frac{1}{2}-\frac{1}{q} 
\Leftrightarrow q>n.
\] 
Thus, the convective terms converge. Also, all coefficients converge thanks to the continuity property of the Nemytskii operators
and the Lebesgue dominated convergence theorem.
For details, see \cite{pemfc}.
Therefore, the limit \((\bm{\Upsilon} ,\Theta, (\phi_{cc}))\)  solves the variational system \eqref{aux-pdi1}-\eqref{aux-phicc}.
  \end{proof}

Finally, the higher integrability of the gradient for  \(\phi_\mathrm{cc}\) is established in \cite{pemfc} as follows, by reproducing the
Gr\"oger elliptic regularity result \cite{gro89,groreh}, applying the limiting current bound \(j_L(=j_{c,L})\) and 
\begin{align*}
M_r& :=\sup\{\|v\|_{1,r,\Omega_a\cup \Omega_c}: \ v\in V_r(\Omega_a\cup \Omega_c), \| Jv\|_{(V_r(\Omega_a\cup \Omega_c))'} \leq 1\} \\
& < \sigma^\#/\sqrt{(\sigma^\#)^2-\sigma_\#^2}  
\end{align*}
for every \(r\geq 2\), where
\[
\langle J\phi,w\rangle =\int_{\Omega_\mathrm{a}\cup \Omega_\mathrm{c}} \nabla\phi\cdot\nabla w\dif{x}.
\]
\begin{proposition}[Regularity]\label{regular}
Let   \(\phi_{cc}\in V(\Omega_p)\) be the solution of the variational equality \eqref{aux-phicc}.
Then, \((\phi_{cc})|_{\Omega_\mathrm{a}\cup\Omega_\mathrm{c}}\)  belongs to the Sobolev space \(W^{1,r}(\Omega_\mathrm{a}\cup\Omega_\mathrm{c})\),
for some  \( r>2 \) depending exclusively on the boundary, and the following quantitative estimate
\begin{equation}\label{cotajoule}
\|\nabla\phi\|_{r,\Omega_a\cup \Omega_c}\leq \frac{\sigma_\#M_r}{\sigma^\#
\left(\sigma^\#-M_r \sqrt{(\sigma^\#)^2-\sigma_\#^2}\right) } j_{L} |\Gamma_{CL}|:=R_3.
\end{equation}
holds.
Moreover, under the conditions of Proposition \ref{propttm},  we have the strong convergence
\(\sigma(\bm{\varrho}_m,\xi_m)|\nabla \phi_m|^2\rightarrow\sigma (\bm{\varrho},\xi)|\nabla\phi|^2\) in \(L^{r/2}(\Omega_\mathrm{a}\cup \Omega_\mathrm{c})\).
\end{proposition}

\begin{remark}\label{regular2}
The domains \(\Omega_\mathrm{a}\)  and \(\Omega_\mathrm{c}\) are regular in the sense in \cite{gro89} for every \(r\geq 2\), which means that the above
regularity is valid for every \(r\geq 2\).
We  define \(t = r/2=2\).
\end{remark}

\section{Fixed point argument(Proof of Theorem \ref{tmain})}
\label{fpthm}

Our aim is to apply the Tychonoff fixed point theorem to the operator \(\mathcal{T}\) defined in \eqref{fpa}.
The closed set \(K\subset E= 
(H(\Omega_p)/\mathbb{R}) \times  [V(\Omega)]^{3}  
\times L^t(\Omega_\mathrm{a}\cup\Omega_\mathrm{c})\),  \(t>1\), defined as
\begin{align*}
K= \lbrace (\pi, \bm{\upsilon},\Phi): \ \| \nabla \pi\|_{2,\Omega_\mathrm{p}}\leq R_1, \ \| \bm{\upsilon} \|_{V_2} \leq R_2,\ 
\|\Phi\|_{t,\Omega_\mathrm{a}\cup\Omega_\mathrm{c}}\leq R_3 \rbrace
\end{align*}
 is compact when the topological vector space is provided by the weak topology, or simply weakly compact, because \(E\) is reflexive.
The radius \(R_1\), \(R_2\)and \(R_3\) are the positive constants defined in \eqref{R1},  \eqref{R2} 
 and  \eqref{cotajoule}, respectively.

The operator \(\mathcal{T}\) is well defined for \(n=2\):
\begin{itemize}
\item  due to Proposition \ref{proppxi}, 
since  \(H(\Omega_{p})\hookrightarrow L^2(\Omega_\mathrm{p})\).
Faced with  \(\min\{\mu_\#/2,\beta_\#\}= \mu_\#/2 >\mu_\#/(2C_K) \), the estimate  \eqref{cotaup} may be rewritten 
\begin{align}\label{cotau2}
  \left( \|  \nabla\mathbf{u} \|_{2,\Omega_\mathrm{f} }^2+ \|\mathbf{u}_T\|_{2,\Gamma} ^2\right)^{1/2} \leq aR_2^2+\sqrt{\frac{2C_K}{\mu_\#} }C_0; \\
\label{R1}
\| \nabla  p\|_{2,\Omega_\mathrm{p}} \leq  \sqrt{\frac{\mu^\#}{K_l} } \left(  \sqrt{2L} \frac{R_M}{\sqrt{\mu_\#}}R_2^2+C_0\right):= R_1 ,
\end{align}
where \(C_0\) is defined in  \eqref{C0} and
\[
a:= 2\sqrt{C_K L}   \frac{R_M}{\mu_\#} . \]

\item  due to Proposition \ref{proptt},
taking  \(\mathbf{w}=\mathbf{u}\in \mathbf{H}^{1}(\Omega_\mathrm{f})\) into  account such that obeys \eqref{cotau2},
 if provided by \eqref{defw}, \textit{i.e.}
\[
aR_2^2 < root_1  - \sqrt{\frac{2C_K}{\mu_\#} }C_0 := b,\]
which is possible by  the smallness condition \eqref{small}. 

\item
 and due to Proposition  \ref{regular} and Remark \ref{regular2},  by taking \(t = 2\).
\end{itemize}
Its continuity results from Propositions \ref{upm} and \ref{propttm}-\ref{regular}.

It remains to prove that \(\mathcal{T}\) maps \(K\) into itself. Let \( ( \pi, \bm{\upsilon}, \Phi)\in K\) be given, and let 
\((p,\bm{\Upsilon},\Theta, |\nabla\phi|_{\Omega_\mathrm{a}\cup\Omega_\mathrm{c}}|^2)=\mathcal{T} (  \pi, \bm{\upsilon}, \xi, \Phi)\).

On the one hand, there exists
 the auxiliary velocity field  \(\mathbf{u}=\mathbf{u}(\pi, \bm{\varrho}|_{\Omega_\mathrm{f}},\xi)\) being in accordance with  Proposition \ref{proppxi}
such that verifies \eqref{cotau2}-\eqref{R1}.
On the other hand, there exists 
\(( \bm{\Upsilon},\Theta, \phi_{cc})=(\bm{\rho},  \theta, \phi) (\mathbf{u},  \bm{\varrho},\xi, \Phi)\)
being in accordance with  Proposition \ref{proptt}.

In order to seek for the existence of \(R_2>0\), we may choose \(R_2\) solving
\begin{equation}\label{R2}
aR_2^2+\sqrt{2C_K / \mu_\#}C_0=root_2 <root_1,\end{equation}
 with
 \(root_2\) being the positive root of the quadratic polynomial
 \begin{equation}\label{2polynomial}
 \min_{i} (a_{i,\#}-a_{i.m}) - (1 /2 + \sqrt{2})\sqrt{L}t-\frac{L}{4}t^2=0.\end{equation}
Then,  the estimate \eqref{cotattphi} may be rewritten as
\begin{equation}
 \sum_{i=1}^2 a_{i,m}  \|\nabla\Upsilon_i\|_{2,\Omega}^2
+ a_{3}\|\nabla\Theta\|_{2,\Omega}^2+  \frac{h_\#}{2} \|\Theta\|_{2,\Gamma_\mathrm{w}}^2 
\leq\frac{(\sigma^\#)^2}{k_\#}  R_3^2 + \frac{h^\#}{2}\|\theta_e\|_{2,\Gamma_\mathrm{w}}^2 +\mathcal{B}_0  . \label{cotarhoitt}
\end{equation}
 
The estimate \eqref{cotarhoitt} yields
\begin{align*}
\sum_{i=1}^2 \|\nabla\Upsilon_i\|_{2,\Omega}^2
+\|\nabla\Theta\|_{2,\Omega}^2+\|\Theta\|_{2,\Gamma_\mathrm{w}}^2 
\leq  \frac{1}{a_\#} \left( \frac{(\sigma^\#)^2}{k_\#}  R_3^2 +c \right)\nonumber \\
 \leq R_2^2 := \frac{root_2 -\sqrt{2C_K / \mu_\#}C_0 }{a}  <\frac{b}{a},
\end{align*}
which is possible by the smallness condition \(root_2 > \sqrt{2C_K / \mu_\#}C_0 + ac/a_{\#}\), that is  \eqref{small}.
Here, we set
\begin{align}
c&:=   \frac{h^\#}{2}\|\theta_e\|_{2,\Gamma_\mathrm{w}}^2 +\mathcal{B}_0 ; \nonumber \\  
a_\#& := \min_{i=1,2,3}\{ a_{i,\#}, a_{i,m}\},\label{defaa}
\end{align}

Therefore, we defined \(R_2\) in such way the smallness condition \eqref{small} is  fulfilled by the data,
which completes the proof of Theorem \ref{tmain}.

\end{document}